\documentclass[12pt,oneside]{amsart}
\usepackage{amscd,a4,amssymb}
\usepackage{typearea,color}
\usepackage{hyperref}

\newcounter{lemmacounter}
\newcounter{thmcounter}

\newtheorem{lemma}[lemmacounter]{Lemma}

\newtheorem{example}[lemmacounter]{Example}
\newtheorem{proposition}[lemmacounter]{Proposition}
\newtheorem*{statementr}{Statement($r$)}
\newtheorem{corollary}[thmcounter]{Corollary}
\newtheorem*{conjecture*}{Conjecture}

\newtheorem{theorem}[thmcounter]{Theorem}
\newtheorem*{theorem*}{Theorem}

\newcommand{\statement}[1]{\textbf{Statement({#1})}}

\newcommand{\fS}{\mathfrak S}

\newcommand{\IR}{{\mathbb R}}
\newcommand{\IC}{{\mathbb C}}
\newcommand{\IRan}{{\mathbb R}_{\mathrm{an}}}

\newcommand{\IRexp}{{\mathbb R}_{\mathrm{exp}}}

\newcommand{\IZ}{{\mathbb Z}}
\newcommand{\IN}{{\mathbb N}}

\newcommand{\IQ}{\mathbb{Q}}
\newcommand{\IQbar}{\overline{\IQ}}

\newcommand{\ssm}{\smallsetminus}

\newcommand{\zeroset}[1]{\mathcal{Z}({#1})}

\newcommand{\mat}[2]{{\mathrm{Mat}}_{#1}({#2})}
\newcommand{\dist}[1]{{\mathrm{dist}}({#1})}
\newcommand{\dists}[1]{{\mathrm{dist}}^{*}({#1})}
\newcommand{\fr}[1]{{\mathrm{fr}}({#1})}

\newcommand{\alg}[1]{{#1}^{\mathrm{alg}}}

\newcommand{\loccit}{\textit{loc.cit.}}
\renewcommand{\subset}{\subseteq}
\renewcommand{\supset}{\supseteq}

\newcommand{\qac}{quasi-algebraic cell}
\newcommand{\qacs}{quasi-algebraic cells}
\newcommand{\QACS}{Quasi-Algebraic Cells}
\newcommand{\Qacs}{Quasi-algebraic cells}

\newcommand{\sing}[1]{\mathrm{Sing}({#1})}

\newcommand{\degvar}{e}

\newcommand{\nbhd}[1]{{\mathcal{N}}({#1})}

\newcommand{\cJ}{\mathcal{J}}

\begin{document}
\title{Diophantine Approximations on Definable Sets}
\author[Philipp Habegger]{P. Habegger}
\address{Department of Mathematics and Computer Science, University of Basel, Spiegelgasse 1, 4051 Basel,
Switzerland}
\email{philipp.habegger\makeatletter @\makeatother unibas.ch}
\date{\today}

\maketitle

\begin{abstract}
Consider the vanishing locus of a 
real analytic function on $\IR^n$ restricted to  $[0,1]^n$.
We bound the number of rational points of bounded height that approximate this
 set very well. 
Our result is formulated and proved in the context of
o-minimal structure which give a general framework to work with
sets mentioned above. 
It complements the theorem of Pila-Wilkie that yields a bound of the
same quality for the number of
rational points of bounded height that lie on a definable set. 
We focus our attention on polynomially bounded o-minimal
structures, allow algebraic points of bounded degree,
and provide an estimate that is uniform over some families of
definable sets. 
We apply these results to  study   fixed length sums of roots of
unity that are small in modulus. 
\end{abstract}

\tableofcontents

\section{Introduction}
\label{sec:intro}

The starting point of our investigation is 
the Counting Theorem \cite{PilaWilkie} of 
Pila and Wilkie in a fixed o-minimal structure.  
In Section \ref{sec:notation} we recall the definition of an o-minimal
structure. 
If not stated otherwise, sets and functions are called definable if they
are definable in this o-minimal structure.
The height of $a/b$ where $a$ and $b$ are coprime integers with $b\ge
1$ 
is $H(a/b) = \max\{|a|,b\}$.  The height of
$(q_1,\ldots,q_n)\in \IQ^n$ is
$\max \{H(q_1),\ldots,H(q_n)\}$ for an integer $n\ge 1$. 
For any subset  $X\subset\IR^n$ we write $\alg{X}$ for the algebraic
locus of $X$, i.e. the union of all connected
real semi-algebraic sets of positive dimension that are contained in
$X$.
Roughly speaking, Pila and Wilkie show  that rational points of bounded height on a definable set
are concentrated on its algebraic locus.

\begin{theorem}[Pila-Wilkie, Theorem 1.8 \cite{PilaWilkie}]
\label{thm:pilawilkie}
Let $X\subset\IR^n$ be a definable set and let $\epsilon >0$. There
exists a constant $c=c(X,\epsilon)>0$ such that
\begin{equation*}
  \#\left\{ q \in (X\ssm \alg{X}) \cap \IQ^n : H(q)\le T \right\} \le
  cT^\epsilon
\end{equation*}
for all $T\ge 1$. 
\end{theorem}

This counting result comes after a long series of work including
papers of
Jarn{\'{\i}}k \cite{Jarnik} and Bombieri-Pila \cite{BombieriPila} 
in the one-dimensional
setting  and Pila \cite{Pila:integerdilation} for certain surfaces. The
counting result was further developed by Pila \cite{Pila:AO} to algebraic points of
bounded height and with  a more precise substitute for $\alg{X}$.
This led to striking applications towards the Andr\'e-Oort Conjecture.

The purpose of this paper is to investigate whether one can find
similar bounds on the number of rational points that
\textit{approximate} a definable set. 

Before we come to our first result, let us introduce some notation. 
We will use $|\cdot|$ to denote the maximum-norm on $\IR^n$. 
For  $\epsilon > 0$ we set
\begin{equation*}
  \nbhd{X,\epsilon} = \{y\in\IR^n: \text{there is $x\in X$ with
    $|y-x|<\epsilon$} \}
\end{equation*}
to be the $\epsilon$ neighborhood of the subset $X\subset\IR^n$.

Let $\IQbar$ denote  the algebraic closure of $\IQ$ in $\IC$. 
The absolute Weil height $H:\IQbar\rightarrow [1,+\infty)$  extends the height defined above from $\IQ$
to $\IQbar$; we give a  precise definition and some basic facts in
 Section \ref{sec:notation}. 
 The height $H(q)$ of $q=(q_1,\ldots,q_n)\in\IQbar^n$ is
$\max\{H(q_1),\ldots,H(q_n)\}$.

Let $T\ge 1$ be a real number and  $e\ge 1$ an integer. We
disregard  algebraic numbers that are not real and set
\begin{equation*}
  \IQ^n(T,\degvar) = \left\{ q\in (\IQbar\cap\IR)^n : H(q)\le T \text{ and
  }[\IQ(q):\IQ]\le \degvar\right\}. 
\end{equation*}
This is a finite set by Northcott's Theorem, see Theorem 1.6.8 \cite{BG}.

An o-minimal structure is called \textit{polynomially bounded} if any definable
function $\IR\rightarrow\IR$ is bounded from above by a polynomial for
all sufficiently large positive arguments, 
cf. Section 4 \cite{DM:96}.

Many results in this paper are restricted to polynomially bounded
o-minimal structures for reasons that will be explained in Example
\ref{ex:nonpoly}.

\begin{theorem}
\label{thm:approx2}
Let $X\subset\IR^n$ be closed and definable in a polynomially bounded o-minimal
  structure. Let $\degvar\ge 1$ be an integer and let $\epsilon>0$.
 There exist
  $c=c(X,\degvar,\epsilon)\ge 1$ and
  $\theta=\theta(X,\degvar,\epsilon)\in (0,1]$ such that if $\lambda\ge
  \theta^{-1}$, then 
\begin{equation*}
   \#\left\{q \in \IQ^n(T,\degvar) \ssm \nbhd{\alg{X},T^{-\theta\lambda}} :  \text{there is }
x\in X\text{ with }|x-q|<T^{-\lambda} \right\}
\le cT^\epsilon
\end{equation*}
for all $T\ge 1$. 
\end{theorem}

Roughly speaking, rational approximations to a definable set
cluster near the algebraic locus. 
Van den Dries  \cite{vdD:TarskiSeidenberg} recognized that
$\IRan$,  the  structure of
restricted real analytic functions, is o-minimal  and even
polynomially bounded using older work of
Gabrielov.


The statement of the theorem simplifies
if  $X$ does not contain a connected real semi-algebraic set of
positive dimension.

\begin{corollary}
\label{cor:noalglocus}
  Let $X\subset\IR^n$ be  closed  and definable in a polynomially
  bounded o-minimal structure such that $\alg{X}=\emptyset$. Let
  $\degvar\ge 1$ be an integer and let $\epsilon>0$. There exist $c=c(X,\degvar,\epsilon)>0$ 
and $\lambda = \lambda(X,\degvar,\epsilon)>0$ such
  that
\begin{equation*}
   \#\left\{q \in \IQ^n(T,\degvar) :  \text{there is }
x\in X\text{ with }|x-q|<T^{-\lambda} \right\}
\le cT^\epsilon
\end{equation*}
for all $T\ge 1$. 
\end{corollary}

Huxley \cite{Huxley:Pisa94} obtained powerful bounds  for the number of rational approximations to the graph of a
 function $\IR\rightarrow\IR$ that is twice  continuously
 differentiable. Here the second derivative is not allowed to vanish.
He made further contributions \cite{Huxley:AA2000} for trice
continuously differential functions.  Huxley's result also covers
 many algebraic functions and  does not distinguish between the
algebraic and transcendental case. 
Applying his result to a graph whose algebraic locus
is empty does not seem to lead to a 
$T^\epsilon$ bound as in Corollary \ref{cor:noalglocus}.

We exhibit  examples which show that some of  the assumptions in our
results cannot be dropped. 

First, let us see why we cannot drop the
hypothesis  that $X$ is closed in Corollary \ref{cor:noalglocus}.
 
\begin{example}
\label{ex:needpb}
 We work  in the o-minimal structure $\IRan$ in which 
    \begin{equation*}
       X = \bigl\{(x,y,(e^{x}-1)(e^y-1)) : x,y\in (0,1) \bigr\} \subset\IR^3.
    \end{equation*}
is definable. It is a  $2$-dimensional cell that is not closed.
There are $3T^2/\pi^2 + o(T^2)$ rational
points $q=(x,0,0)\in [0,1]\times\IR^2$ of height at most $T$, see
 Theorem 330 \cite{HardyWright}.
 Each such point lies in the closure of $X$ in $\IR^3$. 
However, using Ax's Theorem \cite{AxSchanuel} one can show $\alg{X}=\emptyset$. 
So the real semi-algebraic curves in the boundary of a definable set
can lead to many good rational approximations.
\end{example}

If $x\in\IR^n$ and if $X$ is any non-empty subset of $\IR^n$ then we define
\begin{equation*}
  \dist{x,X} = \inf\{ |x-x'| : x'\in X\}
\end{equation*}
and 
\begin{equation*}
  \dists{x,X} = \min \{1,\dist{x,X}\}.
\end{equation*}
It is convenient to define $\dists{x,\emptyset} = 1$ for all $x\in\IR^n$.
The function $x\mapsto \dists{x,X}$ is continuous and it is
definable if $X$ is. 

Second, we construct  an example which shows that
Corollary \ref{cor:noalglocus} is false if we drop the hypothesis that
the o-minimal structure in question  is not
polynomially bounded. 



\begin{example}
\label{ex:nonpoly}
 Set
  \begin{equation*}
    X = \left\{(x,e^{-1/x}) : x\in (0,1] \right\} \cup \left\{(0,0)\right\}
  \end{equation*}
which is definable in $\IRexp$, the structure generated by
the exponential function on the reals, which was proved to be
o-minimal by Wilkie. Observe that $X$ is compact and
$\alg{X}=\emptyset$ as $x\mapsto e^{x}$ is not semi-algebraic. 
For given $\lambda>0$ there is $x_0=x_0(\lambda)\ge 1$ such that 
we have $e^{-x/2}< x^{-\lambda}$ if $x\ge x_0$.
Now let $n\ge 1$ be an integer and suppose $T\ge x_0(\lambda)$. If $T/2\le n$, then
\begin{equation*}
  \left|(1/n,0) - (1/n,e^{-n})\right| = e^{-n} \le e^{-T/2}< T^{-\lambda}.
\end{equation*}
Thus $(1/n,e^{-n})\in X$  approximates
the rational point $(1/n,0)$.
Considering all $n$ with $T/2\le n\le T$ we
find 
\begin{equation*}
\#\left\{q\in \IQ^2 : H(q)\le T \text{ and }\dists{q,X}<T^{-\lambda} \right\}\ge
\frac{T}{2} - 1
\end{equation*}
for all sufficiently large $T$. 
\end{example}

The multiplicative constant in Pila and Wilkie's Theorem is uniform
over
families of definable sets. Somewhat surprisingly, the constant 
$c$ in Theorem \ref{thm:approx2} is not uniform over a definable
family, as we now demonstrate. 

\begin{example}
\label{ex:family}
    We take 
  \begin{equation*}
    Z = \left\{\left(y,x,e^{xy}-1\right) : x,y\in [0,1] \right\} \subset\IR \times\IR^2
  \end{equation*}
and we consider $Z$ as a
definable family parametrized by $y$  with fibers $Z_y$. It is compact and
 definable in $\IRan$.

Observe that $\alg{(Z_y)}=\emptyset$ if $y \in (0,1]$ and 
$\alg{(Z_0)} = Z_0$. In other words, the family $Z$ has transcendental
  fibers away from $0$ which ``degenerate'' to a real semi-algebraic curve
  above $y=0$. This will affect approximation properties of the
transcendental  fibers. 

Let $\lambda >0$,
let $y\in [0,1]$, and suppose $T\ge 1$. For  small
  $y$ there are many ``obvious'' rational points close  to $X_y$ of
  bounded height. 
Indeed, say  $\eta\in\IQ\cap[0,1]$ with $H(\eta)\le T$ then
\begin{equation*}
   |(\eta,0) - (\eta,e^{\eta y}-1)| = e^{\eta y}-1\le 2\eta y\le
  2y
\end{equation*}
as $e^{t}-1\le 2t$ for all $t\in [0,1]$. 
So if $y<  T^{-\lambda}/2$, then as in Example \ref{ex:needpb} we find
\begin{equation*}
  \#\{ q\in \IQ^2 : H(q)\le T \text{ and there is $x\in Z_y$ 
with $|x-q|<T^{-\lambda}$} \} \ge \frac{3}{\pi^2}T^2 + o(T^2)
\end{equation*}
where the constant in $o(\cdot)$ is independent of $y$. 
In particular,   there {cannot} exist  constants $c>0$ and $\lambda >
0$ such that
\begin{equation*}
\#\{ q\in \IQ^2 : H(q)\le T \text{ and }\dists{q,Z_y}<T^{-\lambda}
\}\le cT
\end{equation*}
holds for all $T\ge 1$ and all $y\in[0,1]$ with $\alg{(Z_y)}=\emptyset$. 
  \end{example}

\begin{example}
\label{ex:nonboundedbase}
Here is a variation of the last example. We set
\begin{equation*}
  Z = \left\{\left(y,x,y^{-1}x^{\sqrt 2}\right) : y  \in [1,+\infty) \text{ and }
    x\in [0,1]\right\} \subset\IR \times\IR^2.
\end{equation*}
Then $Z$ is definable in the  structure generated by
$\IRan$ and taking real powers, cf.
the paragraph before Section 3 \cite{DM:96}  and
Miller's paper \cite{Miller:Expansions} for the fact that this
structure is o-minimal and polynomially bounded. 
  This time $Z$ is closed and $\alg{(Z_y)}=\emptyset$ for all $y$. 

Say $\lambda >0$ is arbitrary.
Let $x\in \IQ\cap [0,1]$ with $H(x)\le T$ and $y\ge 1$, then 
\begin{equation*}
\left|(x,0)-\left(x,y^{-1}x^{\sqrt 2}\right)\right|=y^{-1}x^{\sqrt 2}\le y^{-1}.  
\end{equation*}
 If $y > T^\lambda$, then as in Example \ref{ex:family}
\begin{equation*}
\#\left\{ q\in \IQ^2: H(q)\le T \text{ and } \dists{q,Z_y}<T^{-\lambda}
\right\}
\ge \frac{3}{\pi^2} T^2 + o(T^2). 
\end{equation*}
So the constant $c$ in Corollary \ref{cor:noalglocus} is not uniformly
bounded for families of definable sets. 

In this example, the  transcendenal fibers $Z_y$ degenerate to
the line segment $[0,1]\times\{0\}$ as $y\rightarrow +\infty$. 
\end{example}

In order to generalize Corollary
\ref{cor:noalglocus} to a definable family.
 we must make sure that the family contains
no fibers with a non-trivial algebraic locus and that the fibers do
not degenerate into something algebraic at infinity. 
We make these assumptions precise in the next theorem. 
Let $m\ge 0$ be an integer. 


\begin{theorem}
  \label{thm:familybasicversion}
  Let $Z\subset \IR^m\times \IR^n$ be  closed and  definable in a polynomially
  bounded o-minimal structure such that the projection of $Z$ to $\IR^m$ is bounded
  and such that  $\alg{(Z_y)}=\emptyset$ for all $y\in\IR^m$. 
Let
  $\degvar\ge 1$ be an integer and let $\epsilon >0$. There
  exist $c=c(Z,\degvar,\epsilon)\ge 1$ 
and $\lambda = \lambda(Z,\degvar,\epsilon)>0$ such that
\begin{equation*}
   \#\left\{q \in \IQ^n(T,\degvar) :  \text{there is }
x\in X_y\text{ with }|x-q|<T^{-\lambda} \right\}
\le cT^\epsilon
\end{equation*}
for all $T\ge 1$ and all $y\in\IR^m$. 
\end{theorem}

In view of Example \ref{ex:nonboundedbase} we cannot drop the
hypothesis that the projection of $Z$ to $\IR^m$ is bounded in this
last theorem. 

Can one strengthen Theorem \ref{thm:approx2} by
 replacing  $\nbhd{\alg{X},T^{-\theta\lambda}}$ by
$\alg{X}$? The answer is no, as the
following example shows. 

\begin{example}
  Let 
  \begin{equation*}
    \xi = \sum_{n=1}^\infty 10^{-n!}
  \end{equation*}
be Liouville's constant and set
\begin{equation*}
  X = [0,1]\times \{\xi\}.
\end{equation*}
Then $X$ is semi-algebraic, hence definable in any o-minimal structure
and  $X=\alg{X}$. We claim that there
cannot exist constants $\epsilon \in (0,2),c >0,$ and $\lambda > 0$
such that
\begin{equation}
\label{eq:algXapprox}
\#  \left\{ q\in \IQ^2 \ssm \alg{X} : H(q)\le T
\text{ and }\dists{q,X}<T^{-\lambda} \right\} \le cT^\epsilon
\end{equation}
for all $T\ge 2$. 

Indeed, say $\xi_m = \sum_{n=1}^m
10^{-n!}$ for $m\ge 1$. Then $\xi_m\not=\xi$ and $|\xi_m-\xi| \le 2 \cdot 10^{-(m+1)!}$.
Moveover, each $\xi_m$ is rational with height $H(\xi_m) = T$ where $T=10^{m!}$.
For all $x\in\IR$ we have
\begin{equation*}
 |(x,\xi_m)-(x,\xi)| \le 2\cdot 10^{-(m+1)!} = 2 T^{-(m+1)}
< T^{-m}
\end{equation*}
as $T\ge 2$. 
Say $m\ge \lambda$, then $|(x,\xi_m)-(x,\xi)|<T^{-\lambda}$.
As there are $3T^2/\pi^2 + o(T^2)$ rational $x\in [0,1]$ 
with $H(x)\le T$, the  bound  (\ref{eq:algXapprox}) fails for $m$
sufficiently large.  
\end{example}

We now give a variant of Theorem \ref{thm:approx2} which emphasizes points on $X$ that admit a good
rational approximation. We
will deduce all results above  using this point of view.

\begin{theorem}
\label{thm:approx}
Let $X\subset\IR^n$ be closed and definable in a polynomially bounded o-minimal
  structure. Let $\degvar\ge 1$ be an integer and  let $\epsilon >0$. 
There exist  $c=c(X,\degvar,\epsilon)>0$ and
$\theta=\theta(X,\degvar,\epsilon)\in (0,1]$ with
  the following property. 
If $\lambda\ge \theta^{-1}$ and   $T\ge 1$ there exist an integer $N\ge 0$ with $N\le
cT^\epsilon$  and  $x_1,\ldots,x_N\in X$ such that 
\begin{equation}
\label{eq:counting}
   \left\{
   x\in X \ssm \nbhd{\alg{X},T^{-\theta\lambda}}: \text{there is } q \in
   \IQ^n(T,\degvar) \text{ with }  |x-q| <
  T^{-\lambda}  \right\} 
\subset \bigcup_{i=1}^N \nbhd{\{x_i\},T^{-\theta\lambda}}. 
\end{equation}
\end{theorem}

Theorems \ref{thm:approx} and \ref{thm:approx2} are both special cases
of the  next  result. As in Pila and Wilkie's Theorem 1.10
\cite{PilaWilkie} we can replace $\alg{X}$, which need not be
definable, by a subset which is for fixed $T$. Our formulation of the result below
is inspired by Pila's concept of blocks, cf. Theorem 3.6 \cite{Pila:AO}. 
We refer to Section \ref{sec:qac} where some basic definitions involving
real algebraic sets are recalled.

\begin{theorem}
\label{thm:approxqac}  
Let $X\subset\IR^n$ be closed and definable in a polynomially bounded o-minimal
  structure. Let $e\ge 1$ be an integer and let $\epsilon>0$. There exist
  $c=c(X,\degvar,\epsilon)\ge 1, \theta=\theta(X,\degvar,\epsilon)\in (0,1]$, integers
  $l_1,\ldots,l_t\ge 0$ 
and definable sets $D_j \subset \IR^{l_j}\times\IR^n$ for all
$j\in \{1,\ldots,t\}$
with the following properties:
 \begin{enumerate}
 \item [(i)] 
Say $D=D_j$ for some $j\in \{1,\ldots,t\}$ and $z\in
   \IR^{l_j}$. Then  $D_{z}\subset X$ and if $D_z\not=\emptyset$,
then $D_z$ is
 a connected and open subset of the non-singular locus of a real algebraic
set of dimension $\dim D_z$. 
\item[(ii)]
If $\lambda\ge\theta^{-1}$ and   $T\ge 1$ there exist an integer $N\ge 1$ with $N\le cT^\epsilon$
and $(j_i,z_i) \in \{1,\ldots,t\}\times\IR^{l_{j_i}}$ for $i\in
\{1,\ldots,N\}$
such that if
\begin{equation}
\label{eq:approximationC}
x\in X \text{ and } q\in \IQ^n(T,\degvar) \text{ with } 
|x-q|<c^{-1} T^{-\lambda} 
\end{equation}
then $\dists{x,(D_{j_i})_{z_i}} < T^{-\theta\lambda}$
for some $i\in \{1,\ldots,N\}$. 
 \end{enumerate}
\end{theorem}

In Theorem \ref{thm:approxfamilyclosed} below we will state a result for
definable families which, in view of Example \ref{ex:family}, takes  some
additional care to formulate. 

Our argument follows the framework laid out in the proof of Pila and Wilkie of
their counting theorem \cite{PilaWilkie}. 
We use their basic induction scheme, so it is natural to prove
the theorem directly for families of definable sets. 
Moreover, we use their version 
 of the Gromov-Yomdin Reparametrization Theorem in o-minimal structures. 
In order to treat algebraic points that merely approximate a
definable set, we require a suitable
\L ojasiewicz Inequality. However,
even a basic incarnation of this inequality is not
uniform over a definable family, cf. Example
\ref{ex:lojafamilies} below.
This lack of uniformity is ultimately reflected
in Examples \ref{ex:family} and  \ref{ex:nonboundedbase}. 
However, to complete the  induction step we need uniform control over
various quantities attached to fibers of a definable family.  
We resolve this technical difficulty   by introducing a
 uniform substitute for the \L ojasiewicz Inequality, cf.
Proposition \ref{prop:lojasiewicz}. This inequality is the main new ingredient in
this paper.  
 Its proof  requires intricate
 results on o-minimal structures such as the Generic Trivialization
 Theorem. 
Another difference to the original work of Pila-Wilkie, as well as to
 earlier work of Bombieri-Pila \cite{BombieriPila}, is our
construction of the auxiliary function. 
Instead of a Vandermonde Determinant we use an ``approximate
Thue-Siegel Lemma'' to construct the auxiliary function, an idea due
to  Wilkie \cite{W:rationalpoints}. It has the
advantage that we can deal directly with algebraic points of bounded
degree. 

Rational approximations on submanifolds of $\IR^n$ are studied in 
metric diophantine approximation. We mention just a few results and
connections to our work
here. Mahler's influential
problem asked to show that for all $\epsilon > 0$ and all $x\in\IR$ outside a Lebesgue zero
  set,
\begin{equation*}
  \left\{ q \in\IZ :
q\ge 1 \text{ and there exist }p_1,\ldots,p_n\in\IZ \text{ with }
\left| x^i -\frac{p_i}{q}\right|<q^{-1-1/n-\epsilon} \text{ for }1\le i\le n\right\}
\end{equation*}
is finite. Here $(p_1/q,\ldots,p_n/q)$
approximates a point on the curve $\{ (x,x^2,\ldots,x^n)
:x\in\IR\}$ with error $q^{-\lambda}$ where $\lambda =
1+1/n+\epsilon$ is arbitrarily close to the critical value $1+1/n$. 
Sprindzhuk solved Mahler's problem. The
 more general conjecture of 
Baker-Sprindzhuk was proved by Kleinbock and
Margulis. 



In  recent work, Beresnevich, Vaughan, Velani, and Zorin
\cite{BVVZ} obtained upper bounds for the number of sufficiently good rational
approximations on certain submanifolds  in $\IR^n$. 
As in other work mentioned in this direction, 
there is a strong emphasis on the quality of the exponent $\lambda$.

Our  method is of a different nature, it 
yields  little control on this exponent. 
Indeed, $\lambda$ 
produced by Theorem \ref{thm:approx2} comes out of compacity
statements in o-minimality and seems difficult to pin down.
The  trade-off is that  our bounds for the number of rational
approximations grows as an arbitrarily small power of the height. This
has applications, one of which we present here. 

We apply  our results to the 
question of  how small a non-vanishing sum of $n+1\ge 2$ roots of unity can
be. 
This problem appears in connection with eigenvalues of circulant
matrices
in work of Graham and Sloane \cite{GrahamSloane:AntiHad}. 
For an integer $N\ge 1$, Myerson \cite{Myerson:sumofrootsof1} defined $f(n+1,N)$ to be the least positive value of 
\begin{equation*}
  \left|1+\zeta_1+\cdots + \zeta_n\right| \quad\text{where}\quad
  \zeta_1^N = \cdots=\zeta_n^N=1.
\end{equation*}
He proved asymptotic estimates if $n\in
\{1,2,3\}$
for $N$ in certain congruence classes and
 $N\rightarrow +\infty$. 
Here we are interested in lower bounds for $f(n+1,N)$. 
Myerson's result \loccit{} implies 
$f(n+1,N)\ge c N^{-1}$ for some absolute constant $c>0$ in the cases $n=1$ and
$n=2$
and $f(4,N)\ge c N^{-2}$.
A lower bounds that decreases exponentially in $N$
holds by 
Konyagin and Lev's Theorem 1 
\cite{KonyaginLev}. Using Liouville's Theorem from number
theory one finds $f(n+1,N)>(n+1)^{-N}$ in general.
 Upper bounds for 
 $f(n+1,N)$ are discussed in \cite{KonyaginLev,Myerson:sumofrootsof1}
and they decrease polynomially in $N$ for fixed $n$ and large $N$. 
However, it seems to be unknown if a polynomial lower bound
holds if $n\ge 4$. 
The author finds it reasonable to expect the following folklore
conjecture. It would follow from a positive answer to the question
Myerson \cite{Myerson:sumofrootsof1} asks at the end of his paper. 
\begin{conjecture*}
For given $n\ge 1$
there exist constants $c(n)>0$ and $\lambda(n)>0$ such that $f(n+1,N)\ge  c(n)
N^{-\lambda(n)}$
for all $N\ge 1$. 
\end{conjecture*}

We use our result on approximations on definable sets to give some
 give credence to this conjecture.  
Indeed,  we show that set of the prime orders $N=p$ where the
conjecture fails is 
sparse.

\begin{theorem}
\label{thm:pvaries}
For $\epsilon > 0, n\ge 1,$ and $a_0,\ldots,a_n\in\IC\ssm\{0\}$. 
  there exist constants $c=c(a_0,\ldots,a_n,\epsilon)\ge 1$
and $\lambda = \lambda(a_0,\ldots,a_n,\epsilon)>0$ 
such that
\begin{alignat*}1
  \#\{ p \le T \text{ is a prime }: \quad &\text{there are
$\zeta_1,\ldots,\zeta_n\in\IC$ with $\zeta_1^p=\cdots=\zeta_n^p=1$ and
} 
\\ 
&0<\left|a_0+a_1 \zeta_1+\cdots + a_n\zeta_n\right|<
    c^{-1}p^{-\lambda} \}\le cT^\epsilon
\end{alignat*}
for all $T\ge 1$. 
\end{theorem}

We briefly discuss the paper's content. 
In Section \ref{sec:notation} we introduce some common notation. 
Our \L ojasiewicz Inequality is formulated in  Section
\ref{sec:loja}, after that we construct the auxiliary function in
Section \ref{sec:auxiliary}. Section \ref{sec:qac} is a detour on
a class of cells that are locally semi-algebraic  and prove useful in the induction
step. The induction itself is done in Section \ref{sec:induction}
and in Section \ref{sec:proofs} we complete the proofs of the
approximation theorems
mentioned here in the introduction. Section \ref{sec:apps} contains the proof
of Theorem \ref{thm:pvaries} on small sums of roots of unity.

The author is indept to  important suggestions made by Jonathan Pila
at an early stage of this work
and to Felipe Voloch for pointing out a possible
 connection to small sums of roots
of unity. 
 He is grateful to Victor Beresnevich,
David Masser, and Gerry Myerson for comments.
He  thanks Margaret Thomas and Alex Wilkie for their talks
given in Manchester in 2015 and 2013, respectively. 
He also thanks  the Institute for Advanced Study in
Princeton, where this work was initiated at the end of 2013, for its hospitality. While there,
he was supported by the National Science Foundation under
agreement No.~DMS-1128155. Any opinions, findings and conclusions or
recommendations expressed in this material are those of the authors
and do not necessarily reflect the views of the National Science
Foundation.

\section{General Notation}
\label{sec:notation}

The natural numbers are $\IN = \{1,2,3,\ldots\}$ and $\IN_0 = \IN\cup \{0\}$.

Let $n\in\IN$. References to a topology are to the Euclidean topology in $\IR^n$ if
not stated otherwise.
Let $X$ be any subset of $\IR^n$, the closure of $X$ in $\IR^n$ is
denoted
by $\overline X$ and the frontier of $X$ is $\fr{X}=\overline X\ssm
X$. This should not be confused with the boundary of $X$,
 the complement in $\overline X$ of the interior of
$X$. 

We defined the height of a rational number in the introduction. More
generally, if $q\in\IQbar$, then we may proceed as follows.
Let $P\in \IZ[X]$ be the unique irreducible polynomial
with $P(q)=0$ and positive leading coefficient $p_0$. Then 
\begin{equation*}
  H(q) = \left( p_0 \prod_{z\in \IC: P(z)=0} \max\{1,|z|\}
  \right)^{1/\deg P}
\end{equation*}
is the absolute Weil height, or just height, of $q$. The height of a vector in $\IQbar^n$ is the maximal
height of a coordinate. See   Bombieri and Gubler's  Chapter 1.5
\cite{BG} for more details. 
Examples of basic height properties are
\begin{equation}
\label{eq:heightprops}
  H(q+q')\le 2H(q)H(q')\quad\text{and}\quad
H(qq')\le H(q)H(q'). 
\end{equation}

Our  reference for o-minimal structures is van den Dries's 
book \cite{D:oMin}. 
For this paper we use the following straightforward definition.

 A
structure $\fS$ is a sequence $(S_1,S_2,\ldots)$ where each $S_n$ is a
set of subsets of $\IR^n$ such that the following properties hold
true for all $n,m\in \IN$. 
\begin{enumerate}
\item [(i)] The set $S_n$ is closed under taking finite
  unions, finite intersections, and passing to the complement. 
\item[(ii)] If $X\in S_n$ and $Y\in S_m$, then 
$X\times Y\in S_{n+m}$. 
\item[(iii)] If $X\in S_n$  and $n\ge 2$, then the
  projection of $X$ onto the first $n-1$ coordinates lies in
  $S_{n-1}$. 
\item[(iv)] All real semi-algebraic sets in $\IR^n$ lie in $S_n$. 
\end{enumerate}
We call $\fS$ an o-minimal structure if in addition
\begin{enumerate}
\item[(v)] all elements in $S_1$ are finite unions of points and open,
  possibly unbounded,
  intervals.
\end{enumerate}
A set is called definable in $\fS$ if it is a member of some $S_n$.
A function defined on a subset of $\IR^n$ with values in $\IR^m$
is called definable if its graph is in $S_{n+m}$. 
Say $m\in\IN_0$. If $m=0$ we will identify $\IR^m$ with a singleton
and
 $\IR^m\times\IR^n$ with $\IR^n$. 
A definable family, or family parametrized by $\IR^m$, is a definable
subset $Z\subset\IR^{m+n}=\IR^m\times \IR^n$. 
We think of $Z$ parametrizing fibers 
$Z_y = \{x\in \IR^n : (y,x)\in Z\} \subset\IR^n$ 
where $y\in \IR^m$. 
The dimension of a definable set is defined in Chapter 4.1
\cite{D:oMin}; we follow the convention 
$\dim\emptyset=-\infty$. 

If there is no ambiguity about the  ambient o-minimal structure  $\fS$, then 
we call a set or function  definable if it is definable in $\fS$. 

Throughout this paper, we will use some basic properties of o-minimal
structures without mentioning them explicitly. 
For example, if $X$ is definable then so are $\overline X$ and $\fr{X}$, cf.
Lemma 3.4, Chapter 1 \cite{D:oMin}. Moreover, the projection of a
definable set to any collection of the coordinates is again
definable.

  Cells are always assumed to definable in
the ambient o-minimal structure. 
They are the ``building blocks'' of the definable sets,   see Chapter 3
of van den Dries's book \cite{D:oMin}.
 Let us recall some of their properties.
 \begin{enumerate}
\item[(i)]  Cells are non-empty by definition.
\item[(ii)] A cell $C\subset\IR^n$ is a locally closed subset of $\IR^n$, cf. (2.5) in
  Chapter 3 \cite{D:oMin}. So $\fr{C}$ is a closed subset
  of $\IR^n$. 
\item [(iii)] 
Let $C\subset\IR^m\times\IR^n$ be a cell.
If $y\in\IR^{m}$, then the fiber $C_y \subset\IR^m$ is either empty
  or a cell, cf. Proposition 3.5(i) in Chapter 3 \cite{D:oMin}. 
Moreover, the dimension $\dim C_y$ does not depend on $y$ if $C_y\not=\emptyset$. 
We call this value the fiber dimension of $C$ over $\IR^m$. 
\item[(iv)] Suppose $m\ge 1$,
 and write $\pi:\IR^m\times\IR^n\rightarrow\IR^m$ for
the projection onto the first $m$ coordinates of
$\IR^m\times\IR^n$. If $C\subset\IR^m\times\IR^n$ is a cell, then so is
$\pi(C)\subset\IR^m$, cf. (2.8) in Chapter 3 \cite{D:oMin}. 
\end{enumerate}

\section{Variations on \L ojasiewicz}
\label{sec:loja}

Throughout this section we work in a fixed polynomially bounded o-minimal structure.

Here is the prototype of a \L ojasiewicz Inequality for definable
functions.

\begin{theorem}[\L ojasiewicz Inequality]
\label{thm:lojasiewicz}
  Let $X\subset \IR^n$ be a  compact and definable set. Suppose that
  $f:X\rightarrow \IR$ is a continuous and definable function with zero set
  $Z\subset X$. There exist  $c>0$ and a rational number $\delta > 0$ such that 
  \begin{equation*}
    \dists{x,Z} \le c |f(x)|^{\delta}
  \end{equation*}
for all $x\in X$. 
\end{theorem}
\begin{proof}
  This follows from 4.14(2) \cite{DM:96} applied to $f$ and the continuous and
  definable function $g(x) = \dists{x,Z}$. 
\end{proof}

The proof of Pila and Wilkie's Theorem \cite{PilaWilkie} relies on  an inductive
argument. To make the induction step work it
  is necessary to work with families of definable sets and
to bound 
 various quantities attached to the fibers of  the family uniformly. 
Unfortunately, the constants $c$ and $\delta$ in the \L ojasiewicz
Inequality above cannot be choosen uniformly over a definable family. 


\begin{example}
\label{ex:lojafamilies}
  We take $X = [-2,2]\times [-2,2]$ and $f(y,x)=y^2+x^2-1$. The zero set $Z$ of
  $f$ is the unit circle. We consider $X$ and $Z$
  as definable families parametrized by the  coordinate $y$. 
For all $y\in [-2,2]$, the theorem above yields $c_y>0$ and $\delta_y>0$ such
that
\begin{equation*}
  \dists{x,Z_y} \le c_y |f(y,x)|^{\delta_y}
\end{equation*}
for all $x\in Z_y$.

If
  $y<-1$ or $y>1$, then $Z_y=\emptyset$ and by our convention
$x\mapsto \dists{x,Z_y}$ is constant with value $1$ as a function in
  $x\in [-2,2]$. So $c_y |f(y,x)|^{\delta_y}\ge 1$ if $|y|>1$. Now
$|f(y,0)|=|y-1||y+1|$ is arbitrarily small as $y\rightarrow 1$ from
the right. So it is not possible to choose 
 $c_y$ and $\delta_y$  independent of $y$.

Observe that $(y,x)\mapsto \dists{x,Z_y}$ is not continuous
on $X$ as  $\dists{0,Z_y}$ jumps from $1$ to $0$ as $y\rightarrow 1$
from the right. 
\end{example}

The purpose of this section is to prove a suitable substitute for the
\L ojasiewicz Inequality above for a definable family. 

We begin with several preliminary lemmas. 
Recall that the frontier $\fr{X}$ of a set $X\subset\IR^n$ is 
$\overline X \ssm X$. 

\begin{lemma}
\label{lem:borderbound}
  Let $X\subset\IR^n$ be  bounded, locally closed, and
definable and suppose $f:X\rightarrow \IR$ is a
  continuous, definable function with $f(x)\not=0$ for all $x\in X$. 
There exist   $c=c(X,f)> 0$ and a rational number  $\delta=\delta(X,f) > 0$ such that 
\begin{equation*}
  \dists{x,\fr{X}} \le c |f(x)|^{\delta}
\end{equation*}
for all $x\in X$. 
\end{lemma}
\begin{proof}
  We set $g(x) = \dists{x,\fr{X}}$ which yields
a continuous, definable function $g:\overline X \rightarrow \IR$. 
Certainly, $g(x)=0$ for $x\in \fr{X}$. Conversely, if $x\in \overline
X$ and  $g(x)=0$ then we may fix a sequence
$x_1,x_2,\ldots\in \fr{X}$ with limit $x$. 
The frontier $\fr{X}$ is closed in $\IR^n$ as $X$ is locally
closed,  so $x\in \fr{X}$.
Therefore, $g$ vanishes precisely on the frontier $\fr{X}$.

We may apply Lemma
C.8 \cite{DM:96} to $\overline{X}$ and the functions $g$ and $f_1 =
f^{-1}:X\rightarrow\IR$, which are
continuous and definable.
We obtain a definable, continuous, odd, increasing, bijective map
$\phi:\IR\rightarrow\IR$
with $\phi(0)=0$  (as defined on page 512 \cite{DM:96} with $p=0$) such
that
for any $y\in \fr{X}$ we have
$\phi(g(x))/f(x)\rightarrow 0$ if $x\rightarrow y$ with $x\in X$. 
We set $h(x) =\phi(g(x))/f(x)$ if $x\in X$ and $h(x)=0$ if $x\in
\fr{X}$.
Thus $h:\overline X\rightarrow \IR$ is continuous. 

Now $\overline X$ is compact as $X$ is bounded. So there exists $c_1>0$
with $|h(x)|\le c_1$ for all $x\in \overline X$.
Observe that $\phi(g(x))\ge 0$ since $g(x)\ge 0$ and because $\phi$ is odd and increasing. 
Therefore, $\phi(g(x))\le c_1 |f(x)|$ for all $x\in X$.

Finally, as the ambient o-minimal structure is polynomially bounded
 there are constants $c_2>0$ and $\delta > 0$ with 
$\phi(t)\ge c_2 t^{1/\delta}$ for all $t\in [0,1]$. 
We may assume that $\delta\in \IQ$. The lemma follows
 with
 $c = (c_1/c_2)^{\delta}$ since
$g$ takes values in $[0,1]$. 
\end{proof}

The fact that $\delta$ is rational above entails that $t\mapsto
t^{\delta}$ is a definable function. 

We state an easy consequence of Proposition C.13 \cite{DM:96}.
\begin{lemma}
\label{lem:Lojasiewicz1}
  Let $X\subset\IR^n$ be a locally closed, definable set and suppose $f,g:X\rightarrow \IR$ are
  continuous and definable functions  such that $x\in X$ and $f(x)=0$ entails
  $g(x)=0$ and such that $g$ is bounded.
There exists a rational number $\delta=\delta(X,f,g) > 0$ and a
continuous and definable
function $h:X\rightarrow\IR$ with 
\begin{equation*}
  |g(x)|\le  |h(x) f(x)|^{\delta}
\end{equation*}
for all $x\in X$.
\end{lemma}
\begin{proof}
By Proposition C.13 \cite{DM:96} there is $\phi:\IR\rightarrow\IR$ as in the
proof
of Lemma \ref{lem:borderbound} and a continuous, definable function $h:X\rightarrow\IR$ with $\phi(g(x)) = h(x)f(x)$ for all
$x\in X$. 
Observe that $|\phi(g(x))| = \phi(|g(x)|)$. 
The rest of the proof is  now much as the end of the proof
of Lemma \ref{lem:borderbound}. 
\end{proof}

We identify polynomials in $\IR[X_1,\ldots,X_m]$ of degree bounded by
$d\ge 0$ including the zero polynomial with  $\IR^l$ where  $l  = {m+d \choose m}$. 
Thus each $f\in \IR^l$ corresponds to a polynomial in $m$ variables
and we write $\zeroset{f}$ for its  set of zeros in $\IR^m$.

Suppose $m\ge 0$ and let $Z\subset\IR^m\times\IR^n$ be a definable family
parametrized by $\IR^m$.
Let $Y\subset\IR^m$ be the projection of $Z$ to $\IR^m$. It is a
definable set and for $(y,x)\in Y\times\IR^n$ 
\begin{equation*}
  (y,x)\mapsto \inf \{ |x-x'| : x' \in Z_y \}
\end{equation*}
yields a definable function
$Y\times\IR^n\rightarrow \IR$.
So $(y,x)\mapsto \dists{x,Z_y}$ is definable
on $Y\times\IR^n$ and even on $\IR^m\times\IR^n$. 
We cannot expect it to be
 continuous due to Example
\ref{ex:lojafamilies}. 

We come to the first variant of the \L ojasiewicz
Inequality from the beginning of this section.

\begin{lemma}[Flexing]
Let $Z\subset\IR^{l}\times\IR^m\times\IR^n$ be  bounded, definable,
and non-empty. 
 There exist  $c=c(Z) \in (0,1]$, a rational number
$\delta=\delta(Z)>0$, and a compact and definable
 set $Z'\subset\overline Z$
 with $\dim Z'<\dim  Z$  
such that the following property holds. Suppose $f\in \IR^l$,
$y\in\IR^m$,  and 
$x \in Z_{(f,y)}$ such that $|f(x)|\le c$. 
Then 
$\dists{x,Z_{(f,y)}\cap\zeroset{f}} \le  |f(x)|^{\delta}$ or 
$\dists{(f,y,x),Z'} < |f(x)|^{\delta}$. 
\end{lemma}
\begin{proof}
Before this lemma we observed that
\begin{equation*}
\IR^l\times\IR^m\times\IR^n\ni  (f,y,x) \mapsto 
\dists{x,Z_{(f,y)}\cap\zeroset{f}} 
\end{equation*}
  yields a  definable and bounded function $g:Z \rightarrow \IR$. Its values are in $[0,1]$.
  We partition 
 $Z$ into a finite number of cells $C_1,\ldots,C_N
  \subset Z$
such that $g|_{C_i}$ is
continuous for all $1\le i\le N$. 

To prove the lemma it  suffices to prove it in the case $Z=C_i$ for
some $i$. 

Therefore, $Z$ is locally closed and $g$ is continuous. 
We apply Lemma \ref{lem:Lojasiewicz1} to $Z,$ the continuous and
 definable evaluation map
$(f,y,x)\mapsto f(x)$, and $g$ to find
\begin{equation}
\label{eq:disthfbound}
  \dists{x,Z_{(f,y)}\cap\zeroset{f}} \le |h(f,y,x) f(x)|^{\delta_1}
\quad\text{for all}\quad (f,y,x)\in Z
\end{equation}
where  $h:Z\rightarrow\IR$ is  continuous and definable  and $\delta_1
> 0$ is rational.

The sets
\begin{equation*}
  Z_1 = \{(f,y,x) \in Z : |h(f,y,x)| |f(x)|^{1/2} \le
  1\}\quad\text{and}\quad
  Z_2 = Z \ssm Z_1
\end{equation*}
 are definable.

Let  $c \in (0,1]$ and let $\delta>0$ be rational, we will determine them in the argument
below. 
Say $f$ and $y$  are as in the hypothesis and suppose
$x\in Z_{(f,y)}$ with $|f(x)|\le c\le 1$.
There are two cases. 

First let us assume $(f,y,x)\in Z_1$; this includes the case $f(x)=0$.
Then
\begin{equation*}
  \dists{x,Z_{(f,y)}\cap\zeroset{f}} \le |f(x)|^{\delta_1/2} 
\end{equation*}
follows from (\ref{eq:disthfbound}).   The
first possibility in the assertion holds  as we may assume $\delta\le
 \delta_1/2$ and since $|f(x)|\le 1$.

 The second case is $|h(f,y,x)||f(x)|^{1/2} > 1$; in particular $f(x)\not=0$.
Recall that $Z$ is bounded by hypothesis.
 Here we apply Lemma \ref{lem:borderbound} to $Z$
 and the continuous  function $Z\ni (f',y',x')\mapsto
 \max\{1,|h(f',y,'x')|\}^{-1}$ which is continuous, definable,
 and does not attain $0$.
So there is a $\delta_2\in (0,1]$ and $c_1>0$, both independent of
  $f,y,$ and $x$, with
 \begin{equation*}
   \dists{(f,y,x),\fr{Z}} \le c_1 \max\{1,|h(f,y,x)|\}^{-\delta_2}.
 \end{equation*}
We obtain
\begin{equation*}
 \dists{(f,y,x),\fr{Z}} \le c_1
\max\{1,|f(x)|^{-1/2} \}^{-\delta_2}
 = c_1 |f(x)|^{\delta_2/2} \le c_1 c^{\delta_2/4} |f(x)|^{\delta_2/4}.
 \end{equation*}
If $c$ is sufficiently small in terms of $c_1$ and $\delta_2$, then 
$\dists{(f,y,x),\fr{Z}}<|f(x)|^{\delta_2/4}\le 1$. 
We may assume $\delta\le \delta_2/4$, 
so the distance is less than 
$|f(x)|^{\delta}$. 

Now $Z'=\fr{Z}$ is definable and  satisfies $\dim Z' < \dim Z$ by Theorem 1.8 in Chapter 4
\cite{D:oMin}. 
Then  $Z'$ is closed in $\IR^l\times\IR^m\times \IR^n$ and contained
in $\overline{Z}$ as
$Z$ is locally closed. Thus $Z'$ is compact and definable; this concludes the
proof. 
\end{proof}

Next we prove a variant of the H\"older inequality  C.15 \cite{DM:96} 
without a compactness assumption. 
Suppose  $m\ge 1$ and let $\pi:\IR^m\times\IR^n\rightarrow\IR^m$ be
the projection onto the first $m$ coordinates.

\begin{lemma}
\label{lem:hoelder} 
Let $A\subset\IR^m$ be bounded, locally closed, and definable 
and let $K\subset\IR^n$ be compact and definable. 
We  suppose that $\psi:A\times K\rightarrow \IR^k$ is
a  continuous, bounded, definable function. There exist  $c=c(A,\psi)>0$ and rational numbers
 $\delta_{1,2}=\delta_{1,2}(A,\psi)>0$ such that 
  \begin{equation}
\label{eq:case2}
\min\{\dists{\pi(x),\fr{A}},\dists{\pi(y),\fr{A}}\}
\le c\min\left\{1,\frac{|x-y|^{\delta_1}}{|\psi(x)-\psi(y)|^{\delta_2}}\right\}
  \end{equation}
for all $x,y\in A\times K$ with $\psi(x)\not=\psi(y)$.
\end{lemma}

\begin{proof}
Let us abbreviate $X=A\times K$. 
This is a locally closed,  bounded, and definable subset of
$\IR^{m}\times \IR^n$.

  We will apply Lemma \ref{lem:Lojasiewicz1} to 
 $X\times X$ and the functions
 $f(x,y) = |x-y|$ and  $g(x,y) = |\psi(x)-\psi(y)|$. Thus there is  a continuous definable function 
$h:X\times X\rightarrow [0,+\infty)$ and a rational number $\delta_1 =\delta_1(A,\psi)> 0$ 
with 
\begin{equation}
  \label{eq:psihdelta1}
|\psi(x)-\psi(y)|\le h(x,y)|x-y|^{\delta_1}
\end{equation}
for all $x,y\in X$. 

Let us   apply also Lemma
\ref{lem:borderbound} to  $X\times X$. This time we take as  function
$\max\{1,h(x,y)\}^{-1}$, which never vanishes on $X\times X$. We get
constants $c=c(A,f)>0$ and a rational number
$\delta_2=\delta_2(A,f)>0$ with
\begin{equation*}
  \dists{(x,y), \fr{X\times X}} \le c \max\{1,h(x,y) \}^{-\delta_2}
\end{equation*}
for all $x,y\in X$. 
Observe that
\begin{equation*}
  \fr{X\times X} = (\fr{X}\times \overline X)\cup(\overline
X\times\fr{X})= (\fr{A}\times K \times \overline{X}) \cup
(\overline{X}\times \fr{A}\times K)
\end{equation*}
 because $K$ is closed. So
\begin{equation*}
  \dists{(x,y),(\fr{A}\times K \times \overline{X}) \cup
(\overline{X}\times \fr{A}\times K)} \le c \max\{1,h(x,y)\}^{-\delta_2}
\end{equation*}
for all $x,y\in X$. 
The  left-hand side of  is at least
$\min \{\dists{\pi(x),\fr{A}},\dists{\pi(y),\fr{A}}\}$, therefore
\begin{equation}
\label{eq:distsfrXX}
\min \{\dists{\pi(x),\fr{A}},\dists{\pi(y),\fr{A}}\}\le c \max\{1,h(x,y)\}^{-\delta_2}.
\end{equation}
If $x\not=y$ we use 
(\ref{eq:psihdelta1})
to bound the
 right-hand side of (\ref{eq:distsfrXX}) from above. 
Thus
\begin{alignat*}1
  \min\{\dists{\pi(x),\fr{A}},\dists{\pi(y),\fr{A}}\} 
&\le c\max\left\{1,
\frac{|\psi(x)-\psi(y)|}{|x-y|^{\delta_1}}\right\}^{-\delta_2} 
\end{alignat*}
and the lemma follows after adjusting
$\delta_1$ and $\delta_2$. 
\end{proof}

\begin{lemma}[Straightening]
Let $Z\subset\IR^m \times\IR^n$ be compact and definable.
There exist $c=c(Z)\in (0,1]$ and a rational number
$\delta = \delta(Z)>0$ with the following property. 
If $y_0\in \pi(Z)$ and $0 < \epsilon \le c$
 there are $y_1,\ldots,y_N \in \pi(Z)$ with $N\le c^{-1}$
such that 
for any   $p\in Z$ with $|y_0-\pi(p)|\le \epsilon$
 there exist $i\in \{1,\ldots, N\}$ and $x\in Z_{y_i}$ with 
$|(y_i,x)-p|< \epsilon ^{\delta}$. 
\end{lemma}
\begin{proof}
If $Z$ is a finite set
 we can take the $y_i$ to  be all elements in  projection of $Z$ to
$\IR^m$ and $c$ small enough to ensure that
$|y_0-\pi(p)|\le c$ entails $y_0=\pi(p)$. 
In this case we may take $(y_i,x)=p$. 

We now assume $\dim Z\ge 1$. The proof is by induction
where  we suppose that the lemma is proved  in dimensions strictly less than
$\dim Z$. 

Below, the constants $c_{1,2}$ and $\delta_{1,2,3,4}$ are positive and  depend only on $Z$.
The constants $c>0$ and $\delta>0$ from the assertion may depend on them
and will be determined below. 

  By the Generic Trivialization Theorem 
1.2, Chapter 9 \cite{D:oMin} we can
partition 
 each $\pi(Z)$ into finitely many cells
$C_{1}\cup\cdots \cup C_{N}$ such that 
$\pi|_Z:Z\rightarrow\pi(Z)$ is definably
trivializable over each $C_i$. 
We let $\psi_i:C_i\times K_i\rightarrow\pi|_Z^{-1}(C_i)$ denote
the definable homeomorphism coming from a trivialization.
As $K_i$ is homeomorphic to a fiber of $Z\rightarrow\pi(Z)$
it is compact. We will also use the fact that each $C_i$ is locally
closed. 

We may
assume $c\le 1/16$ and $c\le 1/N$.

Say $y_0\in\pi(Z)$ and $0<\epsilon\le c$. For each $1\le i\le N$ we
 choose  auxiliary points
\begin{equation}
\label{eq:chooseyi}
y_i\in C_i \text{ such that } |y_0-y_i|\le \epsilon
\end{equation}
if such an element exists. After renumbering, the $y_i$ will be the 
points in the assertion.

Let $p\in Z$ be as in the hypothesis and suppose $\pi(p)\in C_i$. The
$y_i$ as described above exists and we will prove that there is
$x\in Z_{y_i}$ such  that $|(y_i,x)-p|< \epsilon^{\delta}$.
We are in effect straightening-out the fiber containing the possible 
$p$. 

Observe that if $y_i = \pi(p)$, then we are allowed to choose $x$ with
$p=(y_i,x)$. So let us suppose $y_i\not=\pi(p)$. 
To simplify notation we write $C=C_i,y=y_i,\psi=\psi_i,$ and $K=K_i$. 

There is $z(p)\in K$ with
\begin{equation*}
\psi(\pi(p),z(p)) = p.
\end{equation*}
Recall that $y\not=\pi(p)$, so $\psi(y,z(p))\not=p$. 
The function $\psi:C\times K\rightarrow\IR^n$ takes values  in the 
 bounded set $Z$.
So we may apply lemma \ref{lem:hoelder}  to  $C\times K$ and
$\psi$. We obtain  $c_1>0$ and
 $\delta_{1,2}>0$ such that
\begin{equation*}
  \min\{\dists{ y,\fr{C}},\dists{\pi(p),\fr{C}}\}
\le c_1 \min\left\{1,\frac{| y-\pi(p)|^{\delta_1}}{|\psi(
  y,z(p))-p|^{\delta_2}}\right\};
\end{equation*}
observe that $(y,z(p))$ and $(\pi(p),z(p))$ both lie in $C\times K$
and $\psi(\pi(p),z(p))=p$.

We set $\delta_3 = \min\{1/2,\delta_1/4\}$ and split-up into 2 cases,
the first one being 
\begin{equation}
  \label{eq:straightcase2}
\dists{\pi(p),\fr{C}}\ge \epsilon^{\delta_3}.
\end{equation}

 We recall (\ref{eq:chooseyi}) and the hypothesis $|y_0 - \pi(p)|\le
 \epsilon$ to bound
\begin{equation*}
 | y-\pi(p)| \le | y-y_0| + |y_0- \pi(p)| \le 2\epsilon \le 2c \le 1.
\end{equation*}
So we have either
\begin{equation}
\label{eq:alt1}
  |\psi( y,z(p))-p| < | y-\pi(p)|^{\delta_1/(2\delta_2)}
\le (2\epsilon)^{\delta_1/(2\delta_2)}\le 1 
\end{equation}
or 
$  |\psi( y,z(p))-p| \ge| y-\pi(p)|^{\delta_1/(2\delta_2)}$ and thus
\begin{equation}
\label{eq:alt2}
  \min\{\dists{ y,\fr{C}},\dists{\pi(p),\fr{C}} \}\le c_1 
| y-\pi(p)|^{\delta_1/2} \le c_1 (2\epsilon)^{ \delta_1/2}.
\end{equation}

We can rule out this second possibility. Indeed, if $\dists{ y,\fr{C}} < \epsilon^{\delta_3}/2$, 
then 
\begin{equation*}
 \dists{\pi(p),\fr{C}} < \epsilon^{\delta_3}/2 + |y-\pi(p)|\le
 \epsilon^{\delta_3}/2 + 2 \epsilon.
\end{equation*}
By (\ref{eq:straightcase2}) we find
$\epsilon^{-(1-\delta_3)}<4$ which is a contradiction as $\epsilon\le c\le 1/16$ and $\delta_3\le 1/2$.
So we must have $\dists{ y,\fr{C}}\ge \epsilon^{\delta_3}/2$. 
By (\ref{eq:straightcase2}) the left-hand side of  (\ref{eq:alt2}) is 
at least $\epsilon^{\delta_3}/2$.
This is incompatible with   $0<\delta_3 \le \delta_1/4$ 
 for sufficiently small $c$. 

Therefore,   (\ref{eq:alt1}) holds true.
We may assume $\delta \le\delta_1/(4\delta_2)$,
hence $|\psi( y,z(p))-p| \le (2\epsilon)^{2\delta} <
\epsilon^{\delta}$ because $\epsilon\le c< 1/4$. 
We take $x$ as in $\psi(y,z(p)) = (y,x)$ and this yields the lemma.

The second case is
\begin{equation*}
  \dists{\pi(p),\fr{C}}< \epsilon^{\delta_3}.
\end{equation*}
Observe that $\pi|_Z^{-1}(C)=(C\times\IR^n)\cap Z$ is locally closed in $\IR^m\times\IR^n$
because $Z$ is closed.
Moreover, this preimage is bounded because $Z$ is. 
We can thus  apply
Lemma \ref{lem:borderbound} to $\pi|_Z^{-1}(C)$
and the continuous function $p' \mapsto
\dists{\pi(p'),\fr{C}}$. 
Observe that this function does not vanishes as $\fr{C}$ is closed in $\IR^m$. 
We set $Z' =\fr{\pi|_Z^{-1}(C)}$, so 
\begin{equation*}
  \dists{p,Z'} \le c_2 \dists{\pi(p),\fr{C}}^{\delta_4}
< c_2 \epsilon^{\delta_3\delta_4}.
\end{equation*}
If $c$ is small enough,   the left-hand side 
is strictly less than $1$. So there exists $p'\in Z'$
with $|p-p'|\le c_2  \epsilon^{\delta_3\delta_4}$. 

As in the Flexing Lemma the dimension drops $\dim Z' <
\dim \pi|_Z^{-1}(C)\le \dim Z$. 
The frontier $Z'$ is closed in $\IR^m\times\IR^n$
and definable as $\pi|_Z^{-1}(C)$ is locally closed and definable.
So $Z'$ is compact with  $Z'\subset Z$ because  $Z$ is compact.
Therefore, this lemma holds for the compact and definable set $Z'$ by induction on the
dimension and if $c$ is small enough.
We take as $y_0$ a fixed 
 projection $\pi(p')$ that occurs in this second case. 
We obtain a  point in $Z'$ near $p'$
that is in a bounded number of fibers. 
Both the proximity estimate and the bound on the number of fibers are
sufficient to conclude the lemma for $Z$ in this first case.
\end{proof}

We combine flexing and straightening to prove a 
\L ojasiewicz Inequality for families. 

We will again interpret $\IR^l$ as the vector space of tuples of
polynomials of  degree bounded by $d$.  

\begin{proposition}[\L ojasiewicz in families]
\label{prop:lojasiewicz}
Let $Z\subset\IR^l \times\IR^m \times\IR^n$ be   compact and definable. There exist $c=c(Z)\in (0,1]$ and a rational number
 $\delta=\delta(Z)>0$ with the following property. If
 $f\in\IR^l, y\in\IR^m,$ and $0<\epsilon \le c$, there are
 $(f_1,y_1),\ldots, (f_N,y_N) \in\IR^l\times\IR^m$ with $N\le c^{-1}$
such that for
 all $x\in Z_{(f,y)}$ with $|f(x)|\le\epsilon$  there is
 $i \in \{1,\ldots,N\}$ and $x'\in Z_{(f_i,y_i)}$ with $f_i(x')=0$ and 
$|(f_i,y_i,x')-(f,y,x)| < \epsilon^{\delta}$. 
\end{proposition}

Before we proceed with the proof note that the $(f_i,y_i)$ in the
claim may depend on $(f,y)$, but their number is bounded uniformly. 

\begin{proof}
If $Z$ is finite
we take for the $(f_i,y_i)$ the elements in its projection to
$\IR^l\times\IR^m$
and $c$ small enough to ensure that $x\in Z_{(f,y)}$ and  $|f(x)|\le c$ entail $f(x)=0$. 

 We now assume  $\dim Z\ge
1$. We prove the proposition by induction on $\dim Z$
and suppose that it holds in all dimensions that are strictly less than $\dim Z$. 

Let $x\in Z_{(f,y)}$ with $|f(x)|\le\epsilon\le c$. 
The constants $c_{1,2,3}$ and $\delta_{1,2,3,4}$ below are positive and depend only  on $Z$ but not
on $f,y,x,c,\delta,$ or $\epsilon$. We will fix $c$ and $\delta$ during the
argument below. 

 Let $c_1\in(0,1], \delta_1>0$ be the constants and $Z'\subset
   \overline Z =Z$ the compact
and definable set from the Flexing Lemma applied to $Z$. We may assume
$c\le c_1\le 1/2$. 
According to the Flexing Lemma there are two
cases.

Suppose first $\dists{x,Z_{(f,y)}\cap\zeroset{f}} \le
|f(x)|^{\delta_1} \le \epsilon^{\delta_1} \le c^{\delta_1}< 1$.
Then 
$f(x')=0$ for some $x'\in Z_{(f,y)}$ whose distance to $x$ is at most
$|f(x)|^{\delta_1}\le \epsilon^{\delta_1}$.
In this case the proposition follows as we may assume that $(f,y)$ is among the
$(f_i,y_i)$ and  {$\delta< \delta_1$}.

The second case is when there exits  $(f_0,y_0,x_0)\in Z'$ 
with  distance at most $|f(x)|^{\delta_1}$ to $(f,y,x)\in Z$. 
In particular, $|(f_0,y_0)-(f,y)|\le \epsilon^{\delta_1}$. 
Let us assume that we have found $(f_0,y_0)$ in the projection of $Z'$
that satisfies this inequality.
It will serve as a base point for applying the Straightening
Lemma to $Z'$; we now forget about $x_0$ and $x$. 

Indeed,  suppose $x\in Z_{(f,y)}$ with $|f(x)|\le \epsilon$ is a new point
 that is not covered by the first case. 
Thus there is $(f',y',x')  \in Z'$ with $|(f',y',x')-(f,y,x)|\le
\epsilon^{\delta_1}$. Using the triangle inequality 
we find 
\begin{equation*}
 |(f',y')-(f_0,y_0)|\le  2\epsilon^{\delta_1}\le 2c^{\delta_1/2} \epsilon^{\delta_1/2}
\le\epsilon^{\delta_1/2}
\end{equation*}
 for $c$ sufficiently small. 
After further shrinking $c$ we may
  apply the Straightening Lemma to $Z', \epsilon^{\delta_1/2},$ and $(f_0,y_0)$ when considering
 $\IR^l\times\IR^m\times\IR^n = \IR^{l+m}\times\IR^n$ as
parametrized by $\IR^{l+m}$. 
So $(f',y',x')$ lies at distance strictly less than
$\epsilon^{\delta_1\delta_2/2}$ to a point  $(f_i,y_i,x'')$ in  one
of finitely many fibers $Z'_{(f_i,y_i)}$ of $Z'$ provided for by
Straightening Lemma.
We may assume that 
the number of fibers in question
is at most $c^{-1}$.

By the triangle inequality and the estimates above we get
\begin{alignat}1
\label{eq:distbound}
  |(f_i,y_i,x'')-(f,y,x)| &\le
|(f_i,y_i,x'')-(f',y',x')| + |(f',y',x')-(f,y,x)|\\
\nonumber
&\le \epsilon^{\delta_1\delta_2/2} + \epsilon^{\delta_1}\\
\nonumber
&\le 2 \epsilon^{\delta_3}
\end{alignat}
with $\delta_3 = \delta_1 \min \{1,\delta_2 /2 \}$. 
Now $f_i(x'') = (f_i-f)(x'')+f(x'')$, so 
$|f_i(x'')|\le c_2 \epsilon^{\delta_3} + |f(x'')|$ since $x''$ lies in the
projection of $Z'$ to $\IR^n$, a bounded
 set. Let $\delta_4=\min\{1,\delta_3\}$. By developing $f$ in a series
 around $x$ we find 
 \begin{alignat*}1
|f(x'')| = |f(x + x''-x)| \le  |f(x)| + c_3 |x''-x|\le \epsilon+2c_3 \epsilon^{\delta_3}
\le (1+2c_3)\epsilon^{\delta_4}
 \end{alignat*}
because $f$ lies in the projection of the bounded set $Z$ to $\IR^l$
and since $|x''-x|\le 2\epsilon^{\delta_3}$. Therefore,
$|f_i(x'')|\le (1+c_2+2c_3)  \epsilon^{\delta_4}$.
We may assume {$\delta\le \delta_4/2$}. 
If $c>0$ is sufficiently small, then $\epsilon\le c$ implies
$|f_i(x'')|\le \epsilon^{\delta}$.

Recall that $x'' \in Z'_{(f_i,y_i)}$ and $\dim Z'< \dim Z$. 
The proposition now follows by induction on the dimension combined
with (\ref{eq:distbound}). 
\end{proof}

\section{Construction of the Auxiliary Function}
\label{sec:auxiliary}

 Bombieri and Pila \cite{BombieriPila} and later Pila and Wilkie
 \cite{PilaWilkie}
 use 
the \textit{determinant method} to construct an auxiliary
function. 
Here we use a different approach introduced by Wilkie. He  presented it in his lecture
course at Manchester in 2013 \cite{W:rationalpoints}. 
It is related to the use of the Thue-Siegel
Lemma in transcendence theory. 
Our tool to construct the auxiliary function is Minkowksi's Lattice
Point Theorem.

As in the previous section we let $n\in\IN$ and recall that
$|\cdot|$ denotes the maximum norm on
$\IR^n$.
Below we also use $|\cdot|$ to denote 
the maximum norm of the coefficient vector attached to a
polynomial   in real coefficients and possibly more than one unknown. 
 Moreover, we write $\ell(x) = |x_1|+\cdots + |x_n|$ for
 $x=(x_1,\ldots,x_n)\in\IR^n$. 
For $i = (i_1,\ldots,i_n)\in \IN_0^n$  we set
$x^i = x_1^{i_1}\cdots x_n^{i_n}$ for elements $x_1,\ldots,x_n$ in any
given ring where $0^0$ is interpreted as $1$.
Suppose $k\in\IN$ and let $\phi:(0,1)^k\rightarrow\IR$ be a continuous
function for  which all partial derivatives up-to order $b$ exist. For any
$\alpha=(\alpha_1,\ldots,\alpha_k)\in\IN_0^k$ with $\ell(\alpha)\le b$
we define
\begin{equation*}
  \partial^\alpha \phi = \frac{\partial^{\alpha_1}}{\partial
    X_1^{\alpha_1}}
\cdots
\frac{\partial^{\alpha_k}}{\partial X_k^{\alpha_k}} \phi. 
\end{equation*}
We set $|\phi| = \sup_{x\in (0,1)^k} |\phi(x)|$, which is possibly
$+\infty$.

We begin with some elementary estimates. 

\begin{lemma}
\label{lem:estimate}
\begin{enumerate}
\item [(i)]
  Suppose $x,y\in\IR^n$ and $i\in\IN_0^n$, then 
$|x^{i}-y^i|\le \max\{1,|x-y|\}^{\ell(i)-1} (1+|x|)^{\ell(i)}|x-y|$. 
\item[(ii)] Let $k\in\IN,b\in\IN_0,$ and suppose
  $\phi:(0,1)^k\rightarrow \IR^n$ 
has coordinate functions that have
 continuous parital derivatives up-to order $b$ with
modulus bounded by a real number $B\ge 1$ on $(0,1)^k$. 
 If $i\in\IN_0^k$, then 
  $|\partial^\alpha(\phi^i)|\le B^{\ell(i)} \ell(i)^{\ell(\alpha)}$
for all  $\alpha\in\IN_0^k$ such that $\ell(\alpha)\le b$.  
\end{enumerate}
\end{lemma}
\begin{proof}
For the proof of (i) we write $h=x-y=(h_1,\ldots,h_n)$ and may
  assume $i\not=0$. 
The Binomial Theorem implies
  \begin{equation*}
x^{i}-y^{i} = x^i - (x-h)^i 
= -\sum_{\substack{0\le j_1\le i_1 , \ldots, 0\le j_n\le
    i_n \\ j = (j_1,\ldots,j_n)\not=0}}
{i_1 \choose j_1} \cdots {i_n \choose j_n}
x^{i-j}(-h)^j
  \end{equation*}
where $i=(i_1,\ldots,i_n)$. 
We observe that $|(-h)^j| = |h_1^{j_1}\cdots h_n^{j_n}|\le
|h|^{\ell(j)}\le 
|h| \max\{1,|h|\}^{\ell(j)-1}$
if $j\not=0$. 
Say $x=(x_1,\ldots,x_n)$, then the triangle inequality yields
\begin{alignat*}1
 |x^i-y^i|&\le  |h| \max\{1,|h|\}^{\ell(i)-1}
\sum_{\substack{0\le j_1\le i_1 , \ldots, 0\le j_n\le
    i_n \\ j = (j_1,\ldots,j_n)\not=0}}
{i_1 \choose j_1} \cdots {i_n \choose j_n}
|x^{i-j}| \\
&\le |h| \max\{1,|h|\}^{\ell(i)-1} 
\prod_{k=1}^n \left(\sum_{j_k=0}^{i_k} {i_k \choose j_k}
|x_k|^{i_k-j_k}\right) \\
&= |h|\max\{1,|h|\}^{\ell(i)-1}(1+|x_1|)^{i_1} \cdots(1+|x_n|)^{i_n}
\\
&\le |h|\max\{1,|h|\}^{\ell(i)-1} (1+|x|)^{\ell(i)}
\end{alignat*}
and thus part (i).

For the proof of (ii)
let $\phi_1,\ldots,\phi_d:(0,1)^k \rightarrow\IR$ be continuous
functions for which all partial derivatives up-to order $b$ exist
and are bounded in modulus by $1$. If $\alpha\in \IN_0^k$ with
$1\le \ell(\alpha)\le b$, then using
the Leibniz rule we find
\begin{equation*}
  \partial^\alpha (\phi_1\cdots\phi_d) = 
\sum_{i=1}^d \partial^{\alpha'}\left(\phi_1\cdots
\phi_{i-1}\frac{\partial\phi_i}{\partial X_j} \phi_{i+1}\cdots \phi_d\right)
\end{equation*}
for some $\alpha'\in\IN_0^k$ with $\ell(\alpha')=\ell(\alpha)-1$ if the $j$-th coefficient of
$\alpha$ is non-zero. By induction on $\ell(\alpha)$ we conclude 
$|\partial^{\alpha}(\phi_1\cdots\phi_d)|\le d^{\ell(\alpha)}$. 
The lemma follows after scaling $\phi$ and observing
$\phi^i = \phi_1\cdots\phi_d$ where $d=\ell(i)$ and the
$\phi_1,\ldots,\phi_d$ are certain coordinate functions of $\phi$. 
\end{proof}

The next lemma is a variant of Liouville's Inequality.
For  {$d\in\IN$} we write
\begin{equation*}
  D_n(d) = {n+d \choose n}
\end{equation*}
for the number of monomials in $n$ variables and with degree at most
$d$.

\begin{lemma}
\label{lem:liouville}
  Let $x\in\IR^n$ have algebraic coefficients.
If $f\in \IZ[X_1,\ldots,X_n]\ssm\{0\}$  has degree $d$ and if 
$f(x)\not=0$, then 
  $|f(x)| \ge \left(D_n(d) |f| H(x)^{d n}\right)^{-[\IQ(x):\IQ]}$.
\end{lemma}
\begin{proof}
  Suppose $f(x)\not=0$ and set $K=\IQ(x)$. 
Any maximal ideal $v$ of the ring of integers of $K$
defines a non-Archimedean absolute value $|\cdot|_v$ on $K$ 
with $|p|_v=p^{-1}$ for the prime number $p$ contained in $v$. 
We write $d_v$ for the degree of the completion of $K$ with respect to
$v$ over the field of $p$-adic numbers. 
The product formula,
  cf. Chapter 1.4 \cite{BG},  implies
  \begin{equation*}
\prod_{\sigma:K\rightarrow \IC} |\sigma(f(x))|    \prod_{v} |f(x)|_v^{d_v}  =1
  \end{equation*}
where $\sigma$ runs over all field embeddings and $v$ over all 
maximal ideals as before.

Let  $x=(x_1,\ldots,x_n)$.
For a maximal ideal $v$, the 
ultrametric triangle inequality and the fact that $f$ has integral
coefficients gives
\begin{equation}
\label{eq:nonarchbound}
 |f(x)|_v\le \max\{1,|x_1|_v,\ldots,|x_n|_v\}^d 
\le \max\{1,|x_1|_v\}^d\cdots \max\{1,|x_n|_v\}^d.
\end{equation}

If $\sigma:K\rightarrow\IC$ is a field embedding, then 
\begin{alignat}1
\label{eq:archbound}
  |\sigma(f(x))| &\le D_n(d) |f|
  \max\{1,|\sigma(x_1)|,\ldots,|\sigma(x_n)|\}^d
\\ \nonumber
&\le D_n(d)|f| \max\{1,|\sigma(x_1)|\}^d\cdots
\max\{1,|\sigma(x_n)|\}^d.
\end{alignat}

We take the product of (\ref{eq:nonarchbound}) raised to the $d_v$-th
power over all $v$ and multiply it with the product over all (\ref{eq:archbound}) with
$\sigma$ not the identity. On applying the product formula  we get
\begin{equation*}
  1 \le \left(D_n(d) |f| H(x_1)^d \cdots H(x_n)^d \right)^{[K:\IQ]}|f(x)|,
\end{equation*}
as desired. 
\end{proof}

The key tool for constructing the auxiliary function is the following
``approximate Thue-Siegel Lemma'' which follows from
Minkowski's Lattice Point Theorem. 
We use $|\cdot|_2$ to denote the Euclidean norm on a power of $\IR$. 

\begin{lemma}
\label{lem:approxsiegel}
Let $M,N\in\IN$ with $M\le N$ and suppose $A\in\mat{M,N}{\IR}$ has
rows
$a_1,\ldots,a_M$ 
 with $|a_i|_2\ge 1$ for all $1\le i\le M$. 
We set $\Delta = |a_1|_2\cdots |a_M|_2$.
If $Q\ge 2\sqrt N \Delta^{1/N}$ there exists $f\in\IZ^N\ssm\{0\}$ with 
\begin{equation*}
|f|\le Q\quad\text{and}\quad
|Af| \le (2\sqrt N)^{N/M} Q^{1-N/M} \Delta^{1/M}.
\end{equation*}
\end{lemma}
\begin{proof}
We  set
 $\epsilon =  (2\sqrt{N})^{N/M} Q^{-N/M} \Delta^{1/M}$
which lies in $(0,1]$ by our hypothesis.
The columns of the $(M+N)\times N$ matrix $\overline A$ obtained by augmenting $A$ by  the $N\times N$
unit matrix scaled by $\epsilon$ are a basis of a lattice $\Lambda\subset\IR^{M+N}$ of
rank $N$. Observe that an orthogonal transformation of $\Lambda$ lies
in $\IR^N\times\{0\}$, which we here identify with $\IR^N$. 
By Minkowki's Lattice Point Theorem applied to this transformation
there exists $f\in\IZ^N\ssm\{0\}$ such that
\begin{equation*}
  |\overline A f|_2 \le {2 \det(\overline{A}^{\mathrm{t}}\overline{A})^{1/(2N)}}{\nu_N^{-1/N}}
\end{equation*}
where $\nu_N>0$ is the volume of the unit $N$-ball  in $\IR^N$ and
the determinant  is the  volume of $\Lambda$ squared. 
By the Cauchy-Binet Formula
this determinant  is the sum of the squares of the
determinants of all
 $N\times N$ submatrices of $\overline A$.
 Hadamard's inequality
implies that the absolute value of
each determinant  is at most $|a_1|_2\cdots |a_M|_2 \epsilon^{N-M}=\Delta \epsilon^{N-M}$   since
 $\epsilon \le 1 \le |a_i|_2$ for all $1\le i\le M$. Thus $\det({\overline A}^{\mathrm{t}} \overline A)\le {M+N \choose N}
\Delta^2 \epsilon ^{2(N-M)}\le 4^{N} \Delta^2
\epsilon^{2(N-M)}$ since $M\le N$.
Now $\nu_N = \pi^{N/2}/\Gamma(1+N/2)$ where
$\Gamma(\cdot)$ is the gamma function. 
The inequality $\log \Gamma(x) < (x-1/2)\log(x) -x  + \log(2\pi)/2 + 1 $
 is  well-known for all $x>1$, cf. Lemma 1
\cite{MincSathre}. An elementary calculation yields
$\nu_N^{1/N}\ge 2/\sqrt{N}$. We combine this with the estimates above
and obtain
$|\overline A f|_2 \le 2 \sqrt N \Delta^{1/N}\epsilon^{1-M/N}$. 
Now $|f|\le |f|_2 \le \epsilon^{-1} |\overline A f|_2$ and
$|Af|\le |Af|_2\le |\overline A f|_2$, hence 
\begin{equation*}
  |f|\le 2\sqrt N \Delta^{1/N}\epsilon^{-M/N} = Q
\end{equation*}
by our choice of $\epsilon$ and
$|Af|\le 2 \sqrt N \Delta^{1/N}\epsilon^{1-M/N} 
= \epsilon Q = (2\sqrt{N})^{N/M} Q^{1-N/M}\Delta^{1/M}$, as desired. 
\end{proof}

\begin{proposition}
\label{prop:determinant}
  Let $b,d,k,n,\degvar\in\IN,$ and suppose $D_n(d) \ge (\degvar+1)D_k(b)$. 
 Let $B\ge 1$. 
There exists a constant $c =
c(b,d,k,n,\degvar,B) \ge 1$ with the following property. 
 Suppose $\phi: (0,1)^k\rightarrow\IR^n$ is a map
whose coordinate functions have
 continuous parital derivatives up-to order $b+1$ with
modulus bounded by $B$ on $(0,1)^k$. 
For any real number $T\ge 1$
there exist  $N\in\IN$ with $N\le c T^{{(k+1)n\degvar}\frac{d}{b}}$
and polynomials $f_1,\ldots,f_N \in \IQ[X_1,\ldots,X_n]\ssm
\{0\}$ 
with $\deg f_j\le d$ and $|f_j|=1$ for all $j\in \{1,\ldots,N\}$ such
that the following holds true.
If
\begin{equation*}
  z\in (0,1)^k\text{ and }q\in \IQ^n(T,\degvar)\text{ such that }|\phi(z)-q|<
c^{-1}T^{-\frac{(k+1)n\degvar}{k}\frac{d(b+1)}{b}},   
\end{equation*}
then
$f_j(q) = 0$ and $|f_j(\phi(z))|\le c |\phi(z)-q|$ for some $j\in \{1,\ldots,N\}$. 
\end{proposition}
\begin{proof}
  During the proof of this lemma we will increase $c$ several
  times. 
This constant shall not depend on $T$. Below, $c_1,\ldots,c_6$ are
positive constants that depend on $b,d,k,n,\degvar$ and $B$. We will choose
$c$ in function of these constants. 

For any $i\in \IN_0^n$ with $\ell(i)\le d$ we set $\phi_i(x) =
\phi(x)^i$ for all $x\in (0,1)^k$. We thus get a collection of
$D = D_n(d)$ functions $\phi_i:(0,1)^k\rightarrow\IR$ 
for which all derivatives exist and are continuous up-to order $b+1$. 

Say $T\ge 1$. We take 
\begin{equation}
\label{def:r}
 r =  c' T^{-\frac{(k+1)n\degvar}{k}\frac{d}{b}}\le c'\le 1
\end{equation}
  where $c'\in (0,1]$ is small enough in terms of
  $b,d,k,n,\degvar,B,$ and the $c_i$ and
   is to be
determined. 
Our choice of
$c$  is large enough in terms of $c'$. The hypercube $(0,1)^k$ is
contained in the union of 
\begin{equation}
  \label{eq:Nbound}
N\le  (1+r^{-1})^k\le 2^k r^{-k}
=2^k{c'}^{-k} T^{(k+1)n\degvar \frac{d}{b}}
\end{equation}
 closed hypercubes  of side length  $r$. 

Let $V\subset\IR^k$ be one of these closed hypercubes with $V\cap (0,1)^k\not=\emptyset$.
It will eventually lead to a single polynomial $f=f_V$ as in the
hypothesis. 
As we let $V$ vary over the hypercubes covering  $(0,1)^k$, we 
will get $N$ polynomials. After renumbering them, they will be the
$f_j$ claimed to exist in the assertion of this lemma. 
The estimate for $N$ in the assertion will follow from (\ref{eq:Nbound})
as we may assume  {$c \ge 2^kc'^{-k}$}.

Our approach is to find $D_n(d)$ coefficients $f_i\in\IZ$ for
a polynomial
$f=\sum_{\substack{\ell(i)\le d \\ i=(i_1,\ldots,i_n)}} f_i X_1^{i_1}\cdots X_n^{i_n}$ using Lemma \ref{lem:approxsiegel}.
We  develop the Taylor series of  $p(z) =
f(\phi_1(z),\ldots,\phi_n(z))$ around a fixed 
auxiliary point
$\overline{z}=(\overline{z}_1,\ldots,\overline{z}_k)\in V\cap (0,1)^k$
with Lagrange remainder term.
Indeed, for $z\in (0,1)^k$ we have
\begin{equation}
\label{eq:pztaylor}
  p(z)=\sum_{\substack{\alpha\in\IN_0^k
      \\ \ell(\alpha)\le b}} \left(
\sum_{\substack{i\in\IN_0^n \\ \ell(i)\le d}} f_i
\frac{\partial^\alpha
  \phi^i (\overline{z})}{\alpha!}\right)(z-\overline{z})^\alpha
+\sum_{\substack{\alpha\in\IN_0^k
      \\ \ell(\alpha)= b+1}} \left(
\sum_{\substack{i\in\IN_0^n \\ \ell(i)\le d}} f_i
\frac{\partial^\alpha \phi^i(\xi)}{\alpha!}\right)(z-\overline{z})^\alpha
\end{equation}
where $(\alpha_1,\ldots,\alpha_k)!=\alpha_1!\cdots \alpha_k!$
and where $\xi\in (0,1)^k$ lies on the line segment connecting $\overline z$ and
$z$. We now suppose $z\in V\cap (0,1)^k$, observe that $|z-\overline
z|\le r$. 
We must  find $f_i$ such that  
\begin{equation}
\label{eq:matrixrow}
\frac{r^{-(b-\ell(\alpha))} }{ \left|\left(\frac{\partial^\alpha
  \phi^i(\overline{z})}{\alpha!}\right)_{i\in\IN_0^n, \ell(i)\le d} \right|_2}
 \sum_{\substack{i\in\IN_0^n \\ \ell(i)\le d}} f_i
\frac{\partial^\alpha
  \phi^i (\overline{z})}{\alpha!}
\end{equation}
is small in absolute value for all $\alpha\in\IN_0^k$ with $\ell(\alpha)\le b$; the norm
in the denominator is, as usual, the Euclidean norm. 
The Euclidean norm of the coefficient vector in (\ref{eq:matrixrow})
with respect  to the
 $f_i$  is
$r^{-(b-\ell(\alpha))}\ge 1$.
We thus obtain a matrix  with real coefficients, $D_k(b)$
rows, and $D_n(d)$ columns.  
In order to apply 
 Lemma \ref{lem:approxsiegel} we need to estimate the product $\Delta$
of the Euclidean norms  of the rows of this matrix. This product equals
\begin{equation*}
\Delta 
= r^{-\sum_{\ell(\alpha)\le b} (b-\ell(\alpha))}
=r^{-\sum_{j=0}^b {k+j-1 \choose j} (b-j)}
\end{equation*}
as there are ${k+j-1 \choose j}$ derivatives of precise order $j$. 
We have 
\begin{equation*}
  \sum_{j=0}^b {k+j-1\choose j}(b-j) = \frac{b}{k+1} {b+k
    \choose b} = \frac{b}{k+1} D_k(b)
\end{equation*}
by basic properties of the bionomial coefficients and hence
\begin{equation}
\label{eq:Deltaformula}
  \Delta = r^{-\frac{b}{k+1} D_k(b)}.
\end{equation}

We define
\begin{equation}
\label{def:Q}
  Q = r^{-\frac{b+k+1}{(\degvar+1)(k+1)}} \ge 1.
\end{equation}
Let us verify that $Q$ satisfies the hypothesis of Lemma
\ref{lem:approxsiegel} applied to the $D_k(b)\times D_n(d)$ matrix
constructed above. Indeed,  (\ref{eq:Deltaformula}) implies the
first
 equality in 
\begin{equation*}
  \Delta^{\frac{1}{D_n(d)}} = r^{-\frac{b}{k+1} \frac{D_k(b)}{D_n(d)}}\le
 r^{-\frac{b}{(\degvar+1)(k+1)}} = r^{\frac{1}{\degvar  + 1}-\frac{b+k+1}{(\degvar
     +1)(k+1)}} =r^{\frac{1}{\degvar + 1}} Q
\end{equation*}
the inequality is due to $D_k(b)\le D_n(d)/(\degvar+1)$
and $r\le 1$.
As $r\le c'$ we find $2\sqrt{D_n(d)} \Delta^{1/D_n(d)}\le Q$ if
$c'\le (2\sqrt{D_n(d)})^{-(\degvar +1)}$, which we may assume. 

So there is   $f\in\IZ^{D_n(d)}\ssm\{0\}$ with $|f|\le Q$ such that
(\ref{eq:matrixrow}) is bounded from above in absolute value by
$c_1 Q^{1-D_n(d)/D_k(b)} \Delta^{1/D_k(b)}
\le c_1 Q^{-\degvar} \Delta^{1/D_k(b)}$, we  used $D_k(b)\le D_n(d)/(\degvar+1)$ again.

The terms up-to order $b$ in the Taylor expansion (\ref{eq:pztaylor}) can be bounded as
follows. For any $z\in (0,1)^k \cap V$ we have
\begin{alignat}1
\nonumber
\left|\underbrace{\sum_{\substack{\alpha\in\IN_0^k
      \\ \ell(\alpha)\le b}} \left(
\sum_{\substack{i\in\IN_0^n \\ \ell(i)\le d}} f_i
\frac{\partial^\alpha
  \phi^i (\overline{z})}{\alpha!}\right)(z-\overline{z})^\alpha}_{p_{\text{main}}}\right|
&\le 
\sum_{\substack{\alpha\in\IN_0^k
      \\ \ell(\alpha)\le b}} \left|
\sum_{\substack{i\in\IN_0^n \\ \ell(i)\le d}} f_i
\frac{\partial^\alpha
  \phi^i (\overline{z})}{\alpha!}\right|r^{\ell(\alpha)} \\
\label{eq:taylormainpart}
&\le 
c_1  Q^{-\degvar} \Delta^{1/D_k(b)}  r^{b}
\sum_{\substack{\alpha\in\IN_0^k
      \\ \ell(\alpha)\le b}}
\left(\sum_{\substack{i\in\IN_0^n \\ \ell(i)\le d}}\left(\frac{\partial^\alpha
  \phi^i(\overline{z})}{\alpha!}\right)^2\right)^{1/2},
\end{alignat}
keeping (\ref{eq:matrixrow}) in mind. 
Each Euclidean norm on the right is at most
$D_n(d)^{1/2} B^d d^b$
 by Lemma
\ref{lem:estimate}(ii). Therefore, 
$  |p_{\text{main}}|\le 
c_2  Q^{-\degvar} \Delta^{1/D_k(b)} r^{b}$ where
$c_2 = c_1 D_k(b) D_n(d)^{1/2}  B^d d^b$.
We insert (\ref{eq:Deltaformula}) and obtain
$|p_{\text{main}}|\le c_2 Q^{-\degvar}
r^{b\left(-\frac{1}{k+1}+1\right)} = c_2 Q^{-\degvar}
r^{\frac{bk}{k+1}}$. Next we substitute the expression for
$Q$ from (\ref{def:Q}) to get
\begin{equation}
\label{eq:pmain}
 |p_{\text{main}}|\le c_2 r^{\sigma} \quad\text{with}\quad\sigma= \degvar \frac{b+k+1}{(\degvar+1)(k+1)}+\frac{bk}{k+1}.
\end{equation}

The  remainder in the Taylor expansion (\ref{eq:pztaylor}) can be bounded as follows
\begin{alignat}1
\nonumber
\left|\underbrace{\sum_{\substack{\alpha\in\IN_0^k
      \\ \ell(\alpha)=b+1}} \left(
\sum_{\substack{i\in\IN_0^n \\ \ell(i)\le d}} f_i
\frac{\partial^\alpha
  \phi^i ({\xi})}{\alpha!}\right)(z-\overline{z})^\alpha}_{p_{\text{rem}}}\right|
&\le 
\sum_{\substack{\alpha\in\IN_0^k
      \\ \ell(\alpha)=b+1}} 
\sum_{\substack{i\in\IN_0^n \\ \ell(i)\le d}} 
\left|f_i \frac{\partial^\alpha
  \phi^i ({\xi})}{\alpha!}\right|r^{b+1} \\
\nonumber
&\le {k+b \choose k-1}D_n(d) B^d d^{b+1} |f| r^{b+1} 
\end{alignat}
where we used Lemma \ref{lem:estimate}(ii) again to bound the partial
derivatives of $\phi^i$ at $\xi\in (0,1)^k$. 
We obtain $  |p_{\text{rem}}|\le c_3 |f| r^{b+1}$
where $c_3 = {k+b\choose k-1}D_n(d) B^d d^{b+1}$. 
Observe that (\ref{def:Q}) and the choice of $\sigma$ in (\ref{eq:pmain}) imply
\begin{equation}
  \label{eq:Qrsigma}
Q r^{b+1} =  r^\sigma.
\end{equation}
We recall
$|f|\le Q$ and  find
$  |p_{\text{rem}}|\le c_3 Q r^{b+1}= c_3 r^{\sigma}$ 
with the same exponent as in (\ref{eq:pmain}).
Combining both bounds yields
\begin{equation}
\label{eq:pzbound}
|f(\phi(z))| =   |p(z)| \le |p_{\text{main}}|+|p_{\text{rem}}|
\le c_4 r^\sigma 
\end{equation}
with $c_4= c_2+c_3$.

Now suppose that $q\in(\IQbar\cap\IR)^n$ with $H(q)\le T$ and
$[\IQ(q):\IQ]\le \degvar$ satisfies
\begin{equation*}
  |\phi(z)-q|<  c^{-1} 
T^{-\frac{(k+1)n\degvar}{k}\frac{d(b+1)}{b}} = c^{-1}(r/c')^{b+1}
\end{equation*}
  where $z$ still lies in
$(0,1)^k\cap V$ and where we used (\ref{def:r}). 
We may suppose that $c^{-1}\le {c'}^{b+1}$, hence
$|\phi(z)-q|< r^{b+1} \le 1$.
We note that $f(\phi(z))-f(q)$ is the sum of $D_n(d)$ terms of
the form $f_i(\phi(z)^i-q^i)$ where $\ell(i)\le d$. 
By Lemma \ref{lem:estimate}(i)  we find
$|f(\phi(z))-f(q)|\le D_n(d) |f| (1+B)^d |\phi(z)-q|
\le c_5 Qr^{b+1}$ with $c_5 = D_n(d)(1+B)^d$ as $|f|\le Q$. Using   equality (\ref{eq:Qrsigma}) we
obtain
\begin{equation}
\label{eq:triangle}
  |f(q)-f(\phi(z))|\le c_5 r^{\sigma}.
\end{equation}
Together with  (\ref{eq:pzbound}) we get
\begin{equation}
\label{eq:fqubound}
  |f(q)| \le
|f(q)-f(\phi(z))|+|f(\phi(z))|
\le c_6 r^{\sigma} 
\end{equation}
where $c_6=c_4+c_5$. 

Suppose $f(q)\not=0$. Then we obtain $|f(q)|\ge \left(D_n(d) Q
T^{dn}\right)^{-\degvar}$ from Lemma \ref{lem:liouville}. 
We compare this inequality with (\ref{eq:fqubound}) and rearrange to
get
$D_n(d)^{-\degvar}c_6^{-1} \le r^\sigma Q^\degvar T^{dn\degvar}$. Using
(\ref{def:r}) and (\ref{def:Q}) we find, after a brief calculation, that 
$r^\sigma Q^\degvar T^{dn\degvar} ={c'}^{\frac{bk}{k+1}}$
is
independent of $T$. 
As the exponent  $bk/(k+1)$ of $c'$ is  positive, we arrive at a contradiction for $c'$
sufficiently small. 

So $f(q)=0$. We may replace $f$ by $f/|f|$ to normalize the polynomial.
This yields the first claim of the proposition as the number of $f$ is
bounded by (\ref{eq:Nbound}).

For the second and final claim we will bound $|f(\phi(z))|$ from above.
Now that we have $f(q)=0$ and $|f|=1$ we find as above (\ref{eq:triangle}) that 
\begin{equation*}
  |f(\phi(z))| = |f(\phi(z))-f(q)|  \le D_n(d)(1+B)^d |\phi(z)-q|.
\end{equation*}
We may assume $c \ge D_n(d)(1+B)^d$ and  from this we conclude
 the proof. 
\end{proof}

\section{\QACS}
\label{sec:qac}

In this section, cells are assumed to be definable in a fixed
o-minimal structure which we  do not require to be polynomially
bounded. 
Our ambient o-minimal structure   contains all
semi-algebraic sets which themselves
form an o-minimal structure.
We also often work with
semi-algebraic cells.

Let $n\in\IN$.
We recall the notion of a non-singular point in real
algebraic geometry, our reference is the book of Bochnak, Coste, and
Roy \cite{BCR:RAG}. 
A real algebraic set $A\subset\IR^n$ is the set of common zeros of a
finite number of polynomials in $\IR[X_1,\ldots,X_n]$. Let $0\le r \le
n$ be an integer.
A point $x\in A$ is called non-singular in dimension $r$, if there
exist polynomials $f_1,\ldots,f_{n-r}\in\IR[X_1,\ldots,X_n]$ that
vanish on $A$ and satisfy the rank condition 
\begin{equation*}
\mathrm{Rk} \left(  \frac{\partial f_i}{\partial X_j}(x)\right)_{\substack{1\le
i\le n-r \\ 1\le j\le n}} = n-r
\end{equation*}
and an open neighborhood $U$ of $x$ in $\IR^n$ such
that $A\cap U =  \zeroset{f_1,\ldots,f_{n-r}}\cap U$, see Proposition 3.3.10 \loccit{}
We let
 $\sing{A}$ denote the complement in $A$ of all $x\in A$
that are non-singular in dimension $\dim A$. 
The complement $A\ssm \sing{A}$ is called the non-singular locus of
$A$. It is open in $A$ with respect to the Euclidean and Zariski topologies.
By Proposition
3.3.14 \loccit, $\sing{A}$ is a real algebraic set with $\dim\sing{A}<\dim
A$. The dimension of a real semi-algebraic set as
in \cite{BCR:RAG} coincides with its dimension as a definable set in
an o-minimal structure.

We call a cell of dimension $r$  quasi-algebraic if it is an
open subset of 
 the non-singular locus of a  $r$-dimensional real
 algebraic set.

For example, a $0$-dimensional cell is a \qac{}. An $n$-dimensional cell
in $\IR^n$ is an open subset of $\IR^n$, so it is quasi-algebraic. 

\Qacs{} bare  similarities to Pila's
definable blocks. Indeed, as all cells are connected,
an $r$-dimensional 
\qac{}  is a definable block of dimension of dimension $r$ and degree
$d$ for some $d$ in the sense of Definition 3.4 \cite{Pila:AO}. 
Working with  cells provides advantages  in the
induction step presented in Section \ref{sec:induction} below.

\begin{lemma}
\label{lem:cellinregularset}
Let $C\subset\IR^n$ be a  definable set that is
homoeomorphic 
to an open subset of $\IR^r$, e.g. an $r$-dimensional cell, and 
 contained in a non-empty real algebraic set
$A\subset \zeroset{f_1,\ldots,f_M}$   where 
$f_1,\ldots,f_M \in \IR[X_1,\ldots,X_n]$. Suppose $\dim A = r$ and
\begin{equation}
\label{eq:rankatleastnminusr}
\mathrm{Rk} \left(  \frac{\partial f_i}{\partial X_j}(x)\right)_{\substack{1\le
i\le M \\ 1\le j\le n}} \ge n-r
\end{equation}
for all $x\in C$. Then $C\subset A\ssm\sing{A}$ and $C$ is open in
$A\ssm\sing{A}$. 
If in addition $C$ is an $r$-dimensional cell then it is a \qac. 
\end{lemma}
\begin{proof}
We may assume $r\ge 1$ and $C\not=\emptyset$. Say $x\in C$. The
jacobian matrix
$\left(\partial f_i/\partial X_j (x)\right)_{i,j}$
contains an invertible $(n-r)\times(n-r)$ submatrix. After permuting
coordinates and the $f_i$ we may suppose
\begin{equation*}
\det\left(\frac{\partial f_i}{\partial X_j}(x)\right)_{1\le i,j\le
n-r}\not=0. 
\end{equation*}

Let us define $B=\zeroset{f_1,\dots,f_{n-r}}$, it contains $A$ and $C$. 
By the implicit function theorem, cf. Corollary 2.9.8 \cite{BCR:RAG}, there is an open neighborhood $U$
of $x$ in $\IR^n$ such that $B\cap U$ is homeomorphic to an open
subset of $\IR^r$. By hypothesis,  $C$ is also homeomorphic to an open
subset of $\IR^r$. Observe that $x\in C\cap U\subset B\cap U$. By
invariance of domain, $C\cap U$ is open in $B\cap U$, i.e. 
$C\cap U = B\cap U \cap V$ for an open subset $V\subset\IR^n$.
Recall that $C\subset A\subset B$, 
so $C\cap U \cap V = A\cap U\cap V = B\cap U\cap V$. Therefore, $x\in A$ is non-singular
in dimension $r = \dim A$ and thus
$x\in A\ssm\sing{A}$. Moreover, $x$ lies in $B\cap U\cap V$
which is open in $A$ and contained in $C$. We find that $C$ is open in
$A$ by taking the union of all
$B\cap U \cap V$ as $x$ runs through the
points of $C$.
\end{proof}

If $0\le r\le n-1$ we write $\cJ_{n,r}$ for the set
of subsets $J\subset\{1,\ldots,n\}$ with $\#J=r+1$. 

\begin{lemma}
\label{lem:singulardim}
Suppose that for each  $J\in \cJ_{n,r}$
we are given an irreducible $f_J \in \IR[X_1,\ldots,X_n]$ 
with $\deg_{X_j}(f_J) = 0$ for all $j\in \{1,\ldots,n\}\ssm J$. Then
the set of all $x\in \zeroset{f_J : J \in \cJ_{n,r}}$ 
with 
\begin{equation*}
\mathrm{Rk} \left(  \frac{\partial f_J}{\partial
X_j}(x)\right)_{\substack{
J \in \cJ_{n,r} \\ 1\le j\le n}} < n-r
\end{equation*}
is real algebraic of dimesion at most $r-1$. 
\end{lemma}
\begin{proof}
Let $x$ be as in the hypothesis and let $J\subset \{1,\ldots,n\}$ have
cardinality $r$. For any $i\in \{1,\ldots,n\}\ssm J$ we write
$g_i = f_{J\cup \{i\}}$. 
The $(n-r)\times(n-r)$ diagonal matrix
$\bigl((\partial g_i / \partial X_j) (x)\bigr)_{i,j}$, where
 $i,j\in \{1,\ldots,n\} \ssm J$, is singular by hypothesis. 
So $(\partial g_i / \partial X_i) (x)=0$ for some 
$i$.
The polynomial $g_i$ is irreducible by hypothesis. If $\deg_{X_i}
g_i\ge 1$, then the resultant of $g_i$ and $\partial g_i/\partial
X_i$, taken as polynomials in $X_i$, is a non-zero polynomial 
$h \in \IR[X_j : j \in J]$.  If $\deg_{X_i} g_i = 0$ we set $h =
g_i\not=0$ which only depends on the coordinates in $J$. Observe that
$h(x)=0$ in both cases. 

We have proved that if $x$ is as in the hypothesis, then its
projection to any choice of $r$ coordinates of $\IR^n$ indexed by $J$ lies in the
vanishing locus of a non-zero polynomial in $r$ variables. Therefore, the set of $x$ in
question has dimension at most $r-1$. It is clearly a real algebraic
set. 
\end{proof}

\begin{lemma}
\label{lem:Dalg}
Let $D\subset\IR^n$ be  a connected, definable, 
 open subset of a  real semi-algebraic set. If $\dim D\ge 1$ then $\alg{D}=D$. 
\end{lemma}
\begin{proof}
Say $x\in D$, 
by hypothesis there is an open subset $U\subset\IR^n$
containing $x$ such that $D\cap U$ is semi-algebraic. 
All connected components of $D\cap U$ are semi-algebraic and
 open
in $D\cap U$. So we may suppose that $D\cap U$ contains $x$, is connected,
semi-algebraic, and open in $D$. Now $D\cap U$ cannot be a singleton
since
 $D$ is connected and of positive dimension. So it has positive
dimension and
 $D\cap U \subset \alg{D}$. We conclude $D = \alg{D}$. 
\end{proof}

Let $C\subset\IR^n$ be an $(i_1,\ldots,i_n)$-cell of dimension $r\ge
0$, cf. Section 3.2 \cite{D:oMin} for this terminology. Suppose
$1\le \lambda_1<\lambda_2<\cdots < \lambda_r \le n$ are precisely
those
indices with $i_\lambda = 1$. Let $p:\IR^n\rightarrow \IR^r$
denote the projection onto the coordinates
$\lambda_1,\ldots,\lambda_r$. Then $p|_C$ is injective.

\begin{lemma} 
\label{lem:pullbackcell}
In the notation above suppose $D\subset \IR^r$ is a cell
with $D\subset p(C)$. 
Then $p|_C^{-1}(D)$ is a cell. 
\end{lemma}
\begin{proof}
The proof is by induction on $n$. The case $n=1$ is immediate, so let
us assume $n\ge 2$.
We may also suppose $r\ge 1$. 

We write $\pi:\IR^n\rightarrow\IR^{n-1}$ for the projection onto
the first $n-1$ coordinates and $C'=\pi(C)$. 
We often make use of the fact that $C'$ is an
$(i_1,\ldots,i_{n-1})$-cell
and use other properties listed in Section \ref{sec:notation}. 

If $i_n=0$, then $C$ is the graph of a continuous and definable map
$f:C'\rightarrow\IR$.
We write
$q:\IR^{n-1}\rightarrow \IR^{r}$ for the projection onto the
coordinates $1\le \lambda_1<\cdots <\lambda_r < n$. 
Then $q\circ\pi=p$, so $q(C')=q(\pi(C))=p(C)$ and
\begin{alignat*}1
p|_{C}^{-1}(D)
= \{(x',f(x')) : x'\in q|_{C'}^{-1}(D) \}.
\end{alignat*}
By induction $q|_{C'}^{-1}(D)$ is a cell. This makes $p|_{C}^{-1}(D)$  the graph of a
continuous and definable function over this cell, hence itself a
cell. 

Now say $i_n=1$. Then there are continuous and definable
$f,g:C'\rightarrow\IR$, or $f=-\infty$, or $g=+\infty$, with
$f(x')<g(x')$ for all $x'\in C'$ such that
\begin{equation*}
C = \{(x',t) \in C'\times\IR :  f(x') < t < g(x')\}. 
\end{equation*}
Say $D$ is a  $(j_1,\ldots,j_r)$-cell. 

If $r=1$, then $p:\IR^n\rightarrow\IR$ projects to the final
coordinate and $C = \{\text{point}\} \times \text{interval}$, which is
easy to handle. Say $r\ge 2$. 

Let $p':\IR^{n-1}\rightarrow\IR^{r-1}$ be the projection on the
coordinates
$1\le \lambda_1 <\cdots<\lambda_{r-1} < n$ and
$\pi':\IR^{r}\rightarrow\IR^{r-1}$ onto the first $r-1$ coordinates. 
Then $\pi' \circ p = p'\circ \pi$ and
 $\pi'(D)$ is a cell in $\IR^{r-1}$ contain in $\pi'(p(C))=p'(\pi(C))=p'(C')$. By induction we see that
\begin{equation*}
C''= {p'|_{C'}}^{-1}(\pi'(D)) \subset\IR^{n-1}
\end{equation*}
is a cell.

Say $x'\in C''$. If $t\in \IR$ with
$(p'(x'),t)\in D$, then there is $\widetilde x\in C$ such that
$(p'(x'),t)=p(\widetilde x)$. Observe that $x'\in C'$ and
$\pi(\widetilde x)\in C'$ and that $p'$ is injective on $C'$.
Therefore, $x'=\pi(\widetilde x)$ and so $\widetilde x = (x',t)\in
C$. 

In the first subcase we suppose $j_r=0$. 
Here $D$ is the graph of a suitable $f_D : \pi'(D)\rightarrow\IR$. 
We
have
\begin{equation*}
p|_{C}^{-1}(D) = \{(x',t) : x'\in C'', 
f(x')<t<g(x'), \text{ and } t = f_D(p'(x')) \}. 
\end{equation*}
As we saw in the last paragraph, $x'\in C''$ implies
$f(x')< f_D(p'(x')) < g(x')$, so 
\begin{equation*}
p|_{C}^{-1}(D) = \{(x',t) : x'\in C'' \text{ and } t = f_D(p'(x')) \}
\end{equation*}
is a graph and thus a cell. 

The second subcase is $j_r=1$. Let $f_D,g_D:D'\rightarrow \IR$ with
$f_D<g_D$ on $D'$, or
$f_D=-\infty$, or $g_D=+\infty$ describe the  boundaries for $D$. 
As in the last subcase we find 
\begin{equation*}
p|_{C}^{-1}(D) = \{(x',t) : x'\in C'' \text{ and } f_D(p'(x'))<t<g_D(p'(x')) \}.
\end{equation*}
And so $p|_C^{-1}(D)$ is again a cell. 
\end{proof}

We require the following result of Wilkie. 

\begin{theorem}[Wilkie]
\label{thm:wilkiecovering}
A definable, bounded, open subset of $\IR^n$ is a finite union of open
cells. 
\end{theorem}
\begin{proof}
This is Theorem 1.3 \cite{Wilkie:covering}. The open cells 
 may have non-empty intersection. 
\end{proof}

This theorem extends to cells in the following way. 

\begin{lemma}
\label{lem:coveropenincell}
Suppose $C\subset\IR^n$ is a  cell and let $U\subset C$ be a
bounded and definable set that is open in $C$. 
There exist cells
$C_1,\ldots, C_s\subset \IR^n$, each  of dimension
$\dim C$, with $U=C_1\cup \cdots\cup C_s$ .
\end{lemma}
\begin{proof}
Let $r=\dim C$. 
There is nothing to show if $r=0$, else say $p:\IR^n\rightarrow\IR^r$
is as before Lemma \ref{lem:pullbackcell}. Then
$p|_C:C\rightarrow p(C)$ is a homeomorphism and $p(C)$ is open in
$\IR^r$, cf.  2.7 in Chapter 3 \cite{D:oMin}.
Therefore, $p(U)$ is open in $\IR^r$ and certainly bounded. 
By Wilkie's Theorem above it is covered by  cells that are open in
$\IR^r$. A  cell in such a  covering has dimension $r$
and by Lemma \ref{lem:pullbackcell} its preimage under $p|_{C}$ 
 is again a cell of dimension $r$. 
\end{proof}

\section{Induction Scheme}
\label{sec:induction}

Here is  the main technical result of this paper on diophantine approximation
on definable sets. Our theorems mentioned in the introduction  are
derived from the following statement. 

\begin{theorem}
\label{thm:approxfamilyclosed}
Suppose the ambient o-minimal structure is polynomially bounded. 
 Let $m\in\IN_0,n,\degvar\in\IN,\epsilon > 0$ and suppose
  $Z\subset\IR^m\times\IR^n$ is a closed and definable set whose
 projection to $\IR^m$ is bounded. 
  There exist $c=c(Z,\degvar,\epsilon)\ge 1$, 
$\theta = \theta(Z,\degvar,\epsilon)\in (0,1]$, integers
$l_1,\ldots,l_t\in\IN_0$, and definable sets $D_j \subset
  \IR^{l_j}\times\IR^m\times\IR^n$ for all $j\in \{1,\ldots,t\}$ with the
  following properties:
  \begin{enumerate}
  \item [(i)] Say $D = D_j$ for some $j\in \{1,\ldots,t\},
    z\in\IR^{l_j},$ and $y\in\IR^m$. Then $D_{(z,y)}\subset Z_y$ and
    if $D_{(z,y)}\not=\emptyset$, then $D_{(z,y)}$ is a connected
and  open subset of the non-singular locus of a 
real algebraic set of dimension $\dim D_{(z,y)}$. 
\item[(ii)] If $\lambda\ge \theta^{-1}, y\in\IR^m$ and $T\ge 1$ there exist an integer $N\ge 1$ with 
$N\le c T^\epsilon$ and 
$(j_p,z_p,y_p) \in \{1,\ldots,j\}\times\IR^{l_{j_p}}\times\IR^m$ 
for $p\in \{1,\ldots,N\}$ such that if 
\begin{equation*}
x\in Z_y\text{ and } q\in\IQ^n(T,\degvar)
  \text{ with }|x-q|<c^{-1}T^{-\lambda}
\end{equation*}
then $\dists{x,(D_{j_p})_{(z_p,y_p)}} < T^{-\theta\lambda}$ for some $p\in
\{1,\ldots, N\}$ with $|y-y_p|<T^{-\theta\lambda}$. 
  \end{enumerate}
\end{theorem}

In this section we work in a fixed o-minimal structure  which is
arbitrary at first. 
The goal is to start the induction step and eventually prove Theorem \ref{thm:approxfamilyclosed}.

For the next lemma we do not need to assume that 
the ambient o-minimal structure is polynomially
bounded as in the theorem above. 
Let $n\in\IN$.
For $d\in\IN_0$ we define $\IR[X_1,\ldots,X_n]_d$ to be the vector space of
polynomials
in $\IR[X_1,\ldots,X_n]$ of degree at most $d$ including
$0$. 
We will identify this vector space  with 
 $\IR^{D_n(d)}$.
If $f\in\IR[X_1,\ldots,X_n]$, then  $|f|$ denotes the
maximum norm of the coefficient vector of $f$.  

Let $r\in\IN_0$ with $r\le n -1$. 
Recall that $\cJ_{n,r}$ is the set of subsets of $\{1,\ldots,n\}$ with
$r+1$ elements. 
We define
\begin{alignat*}1
  F_{r,d} &= \left\{ \sum_{J\in \cJ_{n,r}} f_J^2 :  f_J \in \IR[X_j : j
      \in J]_d \text{ and } |f_J|=1
\text{ for all $J\in \cJ_{n,r}$}\right\}
\end{alignat*}
Observe that each $f_J$ 
depends only on the variables indexed by $J$. 
We may identify
$F_{r,d}$ with a subset of $\IR^{{n +2d \choose n}}$.
It is the image of
\begin{equation}
\label{eq:ftupleset}
  \left\{(f_J)_{J\in \cJ_{n,r}} \in \IR[X_1,\ldots,X_n]_d^{{n\choose
      r+1}} :
f_J \in \IR[X_j : j\in J]\text{ and } |f_J| = 1
  \text{ for all }J\in \cJ_{n,r}\right\},
\end{equation}
which we may identify with  a  semi-algebraic subset of $\IR^{{ n+d
    \choose n}{n\choose r+1}}$, 
under the semi-algebraic map $(f_J)_J \mapsto \sum_{J} f_J^2$. Thus $F_{r,d}$ is a semi-algebraic set. 
As this map is  continuous and since (\ref{eq:ftupleset}) is compact,
we conclude that $F_{r,d}$ is compact. 

The zero set $\zeroset{f}\subset\IR^n$ of  $f=\sum_{J} f_J^2$ is the
intersection of the zero sets of  all the $f_J$. The projection of $\zeroset{f}$
to the $r+1$ distinct coordinates in a given $J\in \cJ_{n,r}$ is contained in
$\zeroset{f_J}$, taken as a subset of $\IR^{r+1}$. 
As $f_J\not=0$, this projection does not contain a non-empty
open subset of $\IR^{r+1}$. It follows that $\dim \zeroset {f}\le r$
for all $f\in F_{r,d}$. 

For $n=r$ it is convenient to define $F_{n,d} = \{0\}$ and identify
$0$ with the zero polynomial in $\IR[X_1,\ldots,X_n]$. 
This is clearly also a compact and semi-algebraic set with $\dim
\zeroset{f} \le n$ if $f\in F_{n,d}$. Recall that the fiber dimension
was introduced near the end of Section \ref{sec:notation}. 

\begin{lemma}
\label{lem:countapprox}
 Let $m\in\IN_0,n,\degvar\in\IN,$  and $\epsilon\in(0,1]$. Suppose
 $C\subset\IR^m\times\IR^n$ is a  cell
 whose projection  to $\IR^n$ is  bounded and which has 
fiber dimension $r$
over $\IR^m$. 
There exist constants $c = c(C,\degvar,\epsilon)\ge 1,
d=d(n,r,\degvar,\epsilon)\in\IN$, and $0<\lambda \le 4(r+1)^{2r+2}\degvar^{r+1} (\degvar +1){n\choose r+1}^r\epsilon^{-r}$ with the following
property. Say $y\in\IR^m$. If
$T\ge 1$  there exist $N\in\IN$ with $N\le cT^\epsilon$ and polynomials
$f_1,\ldots,f_N\in F_{r,d}$ such that if 
\begin{equation}
\label{eq:countapproxhyp}
\text{$x\in C_y$ and $q\in\IQ^n(T,\degvar)$ with }
|x-q| <
  c^{-1}T^{-\lambda} 
\end{equation}
then $f_j(q)=0$ and $|f_j(x)| \le |x-q|$ for some $j\in \{1,\ldots,N\}$. 
\end{lemma}
\begin{proof}
Recall that each fiber $C_y\subset\IR^n$ is either
empty or a cell of dimension $r$.  

The case $r=n$ can be handled easily. Indeed, here we may take $c=d=\lambda=1$ and one polynomial
 $f_1 = 0 \in F_{n,1}$ is enough.

Now we assume $r\le n-1$. 
Let $y\in \IR^m$, to prove the lemma we may assume
$C_y\not=\emptyset$.

Let $Z$ denote one of the ${n \choose r + 1}$  projections of
$C$ to $\IR^{m}\times\IR^{r+1}$. Each such projection corresponds to
the choice of  $r+1$ variables among $X_1,\ldots,X_n$.
We let $X'_1,\ldots,X'_{r+1}$ denote these chosen variables.

We define $k = \dim Z_y$ and note that $k\le \dim C_y=r$ as $Z_y$ is
the image of $C_y$ under a projection. 

We proceed by proving the following intermediate claim. 

\textit{
There exist $0<\lambda\le
4(r+1)^{2r+2}\degvar^{r+1} (\degvar +1){n\choose r+1}^r
  \epsilon^{-r}$, $d\in\IN$, 
and $c_1\ge 1$ depending only  on $C,\degvar,$ and
$\epsilon$ with
the following property. 
If $T\ge 1$
there exist $N\in\IN$ with $N\le c_1T^{\epsilon/{n\choose r+1}}$  and
polynomials
 $f'_1,\ldots,f'_N\in \IR[X'_1,\ldots,X'_{r+1}]_d$
with $|f'_1|=\cdots = |f'_N|=1$
 such that if 
\begin{equation}
\label{eq:intermediateclaim}
x'\in Z_y\text{ and }q'\in\IQ^{r+1}(T,\degvar)\cap\IR^{r+1}
\text{ with }
|x'-q'|<c_1^{-1} T^{-\lambda}
\end{equation}
then  $f_j'(q')=0$ and  $|f_j'(x')| \le c_1 |x'-q'|$  for some for $j\in
\{1,\ldots,N\}$.}

We prove the claim in the case  $k=0$  first; here we may take
$\lambda = 2\degvar^2$ and $N=1$. 
The set $Z_y$, being the continuous image of a connected space, is a
singleton $\{x'\}$. 
We fix $q$ in the finite set $\IQ^{r+1}(T,\degvar)\cap\IR^{r+1}$ 
such that $|x'-q|$ is minimal and
 take $f'\in \IR[X'_1,\ldots,X'_{r+1}]$ 
to be the normalization of
$\widetilde{f'}=(X'_1-q_{1})^2+\cdots+(X'_{r+1}-q_{r+1})^2$
where $q=(q_{1},\ldots,q_{r+1})$.
Observe $|\widetilde{f'}|\ge 1$ for all $j$.
If   $q'\in\IQ^{r+1}(T,\degvar)\cap\IR^{r+1}$  is as in (\ref{eq:intermediateclaim})
then 
\begin{equation*}
  |f'(x')|\le |\widetilde{f'}(x')|\le (r+1)|x'-q|^2 \le (r+1)|x'-q'|^2\le (r+1)|x'-q'|
\end{equation*}
by minimality of $|x'-q|$ and since $|x'-q'|\le 1$. 
As we may
assume 
$c_1\ge r+1$ we find $|f'(x')|\le c_1 |x'-q'|$.
It remains to prove that $f'(q')$ vanishes. Note that $|q'-q|\le
|q'-x'|+|x'-q| \le 2|x'-q'|\le 2c_1^{-1}T^{-\lambda}$. 
If $q'\not=q$ then  Liouville's
Inequality, Theorem 1.5.21 \cite{BG}, yields
$|q'-q|\ge (2H(q')H(q))^{-\degvar^2}\ge
2^{-\degvar^2}T^{-2\degvar^2}$. 
Combining upper and lower bound yields
$c_1\le 2^{\degvar^2+1}T^{2\degvar^2-\lambda}$
and so $c_1\le 2^{\degvar^2+1}$ since $\lambda=2\degvar^2$. 
So if we assume, as we may, that $c_1>2^{\degvar^2+1}$, then
$q'=q$. Thus $f'(q')=f'(q)=0$ and 
this settles our intermediate claim  if $k=0$. 

Now say $k\ge 1$. Recall that $k\le r$. 
Let $d$ be an integer satisfying $d+1\ge
(\degvar+1)(r+1)$. We will fix $d$ in terms of $\epsilon$ in a
moment. But first observe that $(\degvar+1)D_k(1) = (\degvar+1)(k+1)\le
(\degvar+1)(r+1)\le 
d+1 \le  D_{r+1}(d)$. The binomial coefficient $D_k(b)$ increases
strictly in $b$  since $k\ge 1$. So there
exists a unique $b\in\IN$, depending on $d$, with
\begin{equation}
\label{eq:Siegelhyp}
  (\degvar+1) D_k(b) \le D_{r+1}(d) < (\degvar+1) D_k(b+1).
\end{equation}
We obtain
\begin{equation*}
\degvar+1>  \frac{D_{r+1}(d)}{D_k(b+1)} 
\ge\frac{D_{k+1}(d)}{D_k(b+1)} 
= \frac{d+1}{k+1}\left( \frac{d+2}{b+2}\cdots
  \frac{d+k+1}{b+k+1}\right)
\ge \frac{d+1}{r+1}\left( \frac{d+2}{b+2}\cdots
  \frac{d+k+1}{b+k+1}\right)
\end{equation*}
and thus we must have $d<b$. Hence each one of the $k$ factors in the parentheses on
the right is greater than $d/b<1$. Therefore, $\degvar+1 > 
(d/b)^r(d+1)/(r+1)$. We rearrange terms and find
\begin{equation}
\label{eq:dbbound}
  \frac db < \left(\frac{(\degvar+1)(r+1)}{d+1}\right)^{1/r}.
\end{equation}
Observe that the right-hand side goes to $0$ as $d$ tends to $+\infty$. 

We choose $d$ 
to be the least integer
$d\ge (\degvar+1)(r+1)-1\ge 1$  such that 
\begin{equation}
\label{eq:epsilonbound}
(k+1)(r+1)\degvar \frac db \le 
\frac{\epsilon}{{n\choose r+1}}
\end{equation}
holds. By rearranging and using $\epsilon\in (0,1]$ as well as $k\le
  r$ we find, using (\ref{eq:dbbound}),  that $d$ satisfies
\begin{equation}
\label{eq:boundd}
  d\le (k+1)^r(r+1)^{r+1}\degvar^r(\degvar +1) {n\choose r+1}^r
  \epsilon^{-r}
\le (r+1)^{2r+1}\degvar^r(\degvar +1) {n\choose r+1}^r \epsilon^{-r}.
\end{equation}
The choice of $d$ uniquely determines  $b$, which is bounded from above in terms of $n,\epsilon,$ and
$\degvar$ only. 

We now apply Pila and Wilkie's reparametrization  Corollary 5.2
\cite{PilaWilkie}. 
Thereby, the fiber 
 $Z_y$ can be covered by the images
of a finite number of maps $\phi:(0,1)^{k}\rightarrow\IR^{r+1}$  for which all 
derivatives up-to order $b+1$ exist, are continuous, and have modulus
bounded by a constant $B\ge 1$. Observe that the number of maps 
 and $B$ are bounded independent of $y$. 
 Pila and Wilkie assume that the definable set is in $(0,1)^{r+1}$,
 but this restriction is harmless as the projection of  $C$ to
 $\IR^n$ is bounded by hypothesis. So we can recover the desired
 statement by scaling. 

We now apply Proposition \ref{prop:determinant} with $n$ replaced by
$r+1$ to the $\phi$, 
recalling (\ref{eq:Siegelhyp}) and  (\ref{eq:epsilonbound}).
For given $T\ge 1$ there is  an integer $N\le c_1 T^{\epsilon / {n\choose r+1}}$,
with $c_1 \ge 1$ as in the said proposition,
and polynomials $f'_1,\ldots,f'_N\in  \IQ[X'_1,\ldots,X'_{r+1}]\ssm
\{0\}$ of degree at most $d$ and norm $1$
such that the assertion of the claim made above holds true
for $\lambda = 4(r+1)\degvar d$ as
\begin{equation*}
  4(r+1)\degvar d\ge 
\frac{(k+1)(r+1)\degvar}{k} \frac{d(b+1)}{b}.
\end{equation*}
Observe that in this case $c_1$ is independent of $y$. 
As $d$ is bounded by (\ref{eq:boundd}) we retrieve
\begin{equation*}
  \lambda \le 4(r+1)^{2r+2}\degvar^{r+1} (\degvar +1){n\choose r+1}^r
  \epsilon^{-r}.
\end{equation*}
This completes the proof of our intermediate claim. 

We may treat the constants $\lambda>0$ and  $c_1\ge 1$ found as independent of the
choice of $r+1$ coordinates. 
The constant in the assertion is 
 $c = \max\{2^n c_1^2,c_1^{n \choose r+1}\}$. 
The construction above yields for each choice of $r+1$ coordinates among all $n$ coordinates of
$\IR^n$, given $T$, a tuple of at most $c_1 T^{\epsilon /  {n\choose r+1}}$
normalized polynomials in the corresponding $r+1$ variables and 
with the stated properties. 
We take as the $f_j$ all possible  sums of squares
of the $f'_j$ that appear above where each term corresponds to one of
the ${n \choose r+1}$ projections. 
In total there at most $c_1^{n \choose r+1} T^\epsilon\le cT^\epsilon$ possible
polynomials by our choice of $c$,  
they lie in $F_{r,d}$

Now say
$x\in C_y$ and $q\in \IQ^n(T,\degvar)$ with
 $|x-q|<c^{-1}T^{-\lambda}\le
(2^n c_1^2)^{-1}$.
Then one of the $f_j$ just constructed satisfies $f_j(q)=0$ and
\begin{equation*}
  |f_j(x)|\le {n\choose r+1} c_1^2 |x-q|^2 
\le 2^n c_1^2 |x-q||x-q|
 \le |x-q|. \qedhere
\end{equation*}
\end{proof}

The coefficients of each polynomial $f_j$ produced by this last lemma
are algebraic and have uniformly bounded degree over $\IQ$. 

The fact that some $f_j$ vanishes at $q$ will
play no role in the remaining argument. But from this conclusion we  
can infer something about
 algebraic approximations of a bounded cell $C$
without restricting to polynomially bounded sets. 
Indeed, they
 lie on at most $cT^\epsilon$  real algebraic sets of dimension at
 most $\dim C$ that are cut out by a polynomial of controlled degree.

For the rest of this section we suppose that the ambient o-minimal
structure is polynomially bounded.

The following statement is proved by induction on 
the fiber dimension $r\in \IN_0$. In the induction step we need to keep track of
additional data, for this reason we work with a prescribed cell
partition 
 of our given definable family.

\begin{statementr}
 Let $m\in \IN_0,n,\degvar\in\IN$ with $r\le n,$ and  let $\epsilon
 \in (0,1], \kappa \in (0,1]$. Suppose we are given
 $(Z,C_1,\ldots,C_s)$ where 
 $Z\subset \IR^m\times \IR^n$ is  compact and definable such that
$C_1\cup \cdots\cup C_s$   is a  partition
of $Z$ into 
 cells $C_1,\ldots,C_s$. There exist
 $c=c(C_1,\ldots,C_s,\degvar,\epsilon,\kappa)\ge 1$,  
$\theta=\theta(C_1,\ldots,C_s,\degvar,\epsilon)\in (0,1]$,
integers $l_1,\ldots,l_t\in\IN_0$, 
and bounded cells $D_j \subset \IR^{l_j}\times\IR^m\times\IR^n$ for all
$j\in \{1,\ldots,t\}$
with the following properties:
 \begin{enumerate}
 \item [(i)] Say $D= D_j$ for some $j\in \{1,\ldots,t\}, z\in
   \IR^{l_j},$ and $y\in \IR^m$. 
Then $D_{(z,y)}\subset Z_y$ and if 
$D_{(z,y)}\not=\emptyset$, then
$\dim D_{(z,y)} \le r$ and $D_{(z,y)}$ is a \qac.
\item[(ii)]
Say $C=C_j$ has fiber dimension $r$ over $\IR^m$. 
If $\lambda\ge \theta^{-1},y\in\IR^m,$ and $T\ge 1$ there exist an
integer $N\ge 1$ with $N\le
cT^{\epsilon}$ and  $(j_p,z_p,y_p) \in 
\{1,\ldots,t\}\times\IR^{l_{j_p}}\times\IR^m$
for $p\in \{1,\ldots,N\}$ such that if 
\begin{equation}
\label{eq:approximationC}
 x\in C_y \text{ and }q\in\IQ^n(T,\degvar)
 \text{ with }|x-q|<c^{-1}T^{-\lambda}
\end{equation}
then $\dists{x,(D_{j_p})_{(z_p,y_p)}} < \kappa T^{-\theta\lambda}$
for some $p\in \{1,\ldots,N\}$ with $|y-y_p|<\kappa T^{-\theta\lambda}$. 
 \end{enumerate}
\end{statementr}
\begin{proof}
  We prove by induction on $r$ that \statement{$r$} holds true for
all $r$. 
During the argument we will choose  $c\ge 1$ and 
 $\theta > 0$ in terms of the appropiate data.

If $r=0$ and if   $C$ is a cell  appearing  in (ii)
then any  non-empty
fiber $C_y\not=\emptyset$  consists of a single
point. Therefore, \statement{$0$} holds true by taking the $D_j$ to equal the
$C_j$ that have  fiber dimension $0$ over $\IR^m$ and  $l_j=0$. 
Part (ii) follows with $N=\theta=c=1$ and $y_1=y$.
\statement{$n$} can be handled in a similar fashion. It holds true by adding those $C_j$ to our list in
(i) that have fiber dimension $n$ over $\IR^m$; indeed, $n$-dimensional cells are quasi-algebraic.

So let $1\le r\le n-1$ and 
 suppose that \statement{$r'$} holds true for all $r'\le r-1$.

Let $C,\lambda,y,$ and $T$ be as in (ii). We apply Lemma
\ref{lem:countapprox} to $C$ and obtain  $c_1,d,$ and $\lambda_1$. We
may assume that $\lambda_1$ attains the upper
bound provided by the lemma, so it depends only on $n,\degvar, r,$ and
$\epsilon$.
We may also suppose
$c\ge c_1$ and $\theta \le \lambda_1^{-1}$. So
$\lambda \ge \theta^{-1}\ge \lambda_1$ and hence
$c^{-1}T^{-\lambda}\le c_1^{-1}T^{-\lambda_1}$.
By the lemma there is 
a collection $f_1,\ldots,f_U$ of polynomials in $F_{r,d}$  with $U\le c_1T^\epsilon$
such that any pair $q,x$ as in 
(\ref{eq:approximationC}) 
satisfies $|f_j(x)|\le|x-q| < c^{-1} T^{-\lambda}$ for some $j\in \{1,\ldots,U\}$.

Recall that $F_{r,d}$ is a compact real
semi-algebraic set and that $\overline C$ is the closure  in
$\IR^m\times\IR^n$ of  a 
bound cell. Therefore, 
$F_{r,d}\times\overline C$ is compact and definable.
We will apply Proposition \ref{prop:lojasiewicz} to
 $\epsilon$ replaced by $c^{-1} T^{-\lambda}$ and to
$F_{r,d}\times\overline C$. 
After increasing $c$  
we can  make $c^{-1}T^{-\lambda} \le c^{-1}$ smaller than
$c_2^{-1}$ where 
  $c_2 =
c_2(C,r,d)>0$ is from the said proposition. 
Each
 $f_j$ from above leads to at most $c_2^{-1}$ new elements in
$F_{r,d}$. 
By abuse of notation let us also call them $f_1,\ldots, f_U$ after
renumbering; 
we have $U\le c T^\epsilon$
as we may suppose $c\ge c_1c_2^{-1}$. 
Observe that these new
polynomials approximate the original ones and could now have
transcendental coordinates.
Being in $F_{r,d}$, each $f_j$ is a sum $\sum_{J\in \cJ_{n,r}} f_{j,J}^2$ 
where $f_{j,J}$ depends only on the $r+1$ variables associated to $J$. 
The number of terms is ${n\choose r+1}$ and $\deg f_{j,J}\le d$. 
We split each $f_{j,J}$ into irreducible factors. So after replacing $c$ by a possibly larger constant
we may assume that   $U\le cT^\epsilon$ and that each $f_{j,J}$ is
irreducible with $|f_{j,J}|=1$.

Let  $\delta=\delta(C,r,d)>0$ also come from Proposition
\ref{prop:lojasiewicz}. 
This  proposition  yields $y_1,\ldots,y_U$
with $U\le c T^\epsilon$ 
such that the following holds. 
For any $x$ as above
there is $j$ and $x'\in \overline{C}_{y_j} \cap \zeroset{f_j}$
with 
\begin{equation}
\label{eq:distxy2}
  \max\{|x'-x|,|y_j-y| \} < c^{-\delta}T^{-\delta\lambda}
 \le \frac{\kappa}{2} T^{-\delta\lambda}
\end{equation}
as we may assume $c^\delta \ge 2/\kappa$.

The point $\left((f_{j,J})_{J\in \cJ_{n,r}},y_j,x'\right)$ is
a member of
the compact and definable set
\begin{equation}
\label{eq:defineZprime}
Z'=  \left\{\left((f_{J})_{J\in\cJ_{n,r}},y',x''\right) \in 
\IR[X_1,\ldots,X_n]_{d}^{n\choose r+1} \times\overline C : f_J(x'') = 0\text{ and
  $|f_J|=1$ for all $J\in \cJ_{n,r}$} \right\}.
\end{equation}
Observe that each fiber $Z'_{((f_J)_J,y')}$ is contained in 
$\zeroset{(f_J)_{J\in \cJ_{n,r}}}$ which is a real algebraic set of dimension at most
$r$ by the remark below (\ref{eq:ftupleset}). To avoid singularities we introduce the subset
\begin{equation}
\label{eq:defZprimeprime}
Z''=  \left\{\left((f_{J})_{J\in\cJ_{n,r}},y',x''\right) \in 
Z' :  \mathrm{Rk} \left(\frac{\partial f_J}{\partial
  x_j}(x'')\right)_{\substack{J\in \cJ_{n,r}\\1\le j\le n}}
< n-r \right\}.
\end{equation}
which is again compact and definable. 

We fix a cell partition $D_1\cup \cdots \cup D_{t''}=Z''$ and a cell
partition
$D_{t''+1}\cup\cdots \cup D_{t''+t'} = Z'\ssm Z''$.   
So $D_1\cup\cdots \cup D_{t''+t'}$ is a partition of $Z'$ into cells. 
Note that each cell is bounded since $Z'$ is compact. 

The point  $((f_{j,J})_{J\in \cJ_{n,r}},y_j,x')$ from Proposition \ref{prop:lojasiewicz} lies
in one of these cells, $D$, say. As already pointed out above, we have
\begin{equation}
\label{eq:dimDbound}
\dim D_{((f_{j,J})_J,y_j)} \le \dim Z'_{((f_{j,J})_J,y_j)} \le \dim
\zeroset{(f_{j,J})_{J}} \le r.
\end{equation}
We split up into two cases depending on the value of 
$r' = \dim D_{((f_{j,J})_J,y_j)}$. 

First, suppose $r' \le r-1$. In this case, we consider
$Z'$ as a definable set parametrized by $\IR^{m'}$, where
$m'= {{n\choose r+1}{n + d \choose n}+m}$. 
We can thus apply \statement{$r'$} to 
$(Z',D_1,\ldots,D_{t''+t'})$ and $\degvar,\epsilon,\kappa/2$ to
 obtain $c'$ and $\theta'$. The point $x'$
lies in a  fiber of the cell $D$.
Moreover, as
$|x-q|<c^{-1}T^{-\lambda}$, we get 
\begin{equation*}
\label{eq:xprimeminusq}
  |x'-q| \le  |x'-x|+|x-q|< c^{-\delta}  T^{-\delta\lambda}+c^{-1}T^{-\lambda}
\end{equation*}
using the first inequality of (\ref{eq:distxy2}).
We are free to increase $c$ and decrease $\theta$ to   assume 
$c^{-\delta}+c^{-1}\le {c'}^{-1}$ and $\theta\le \theta'\min\{1,\delta\}$,
respectively. 
As $T\ge 1$, the right-hand
side of (\ref{eq:xprimeminusq}) is
at most $ {c'}^{-1}T^{-\lambda'}$ where $\lambda' =
\min\{1,\delta\}\lambda$.
Observe that
$\lambda' \ge \min\{1,\delta\} \theta^{-1} \ge \theta'^{-1}$.
By induction we find that $x'$ 
has distance at most $\frac{\kappa}{2}T^{-\theta' \lambda'}$ to 
 the union of at most $c' T^\epsilon$ fibers of one of
finitely many bounded cells $D'' \subset
\IR^{l'}\times\IR^{m'}\times\IR^n$. 
More precisely, we have 
\begin{equation*}
  \max\{\dists{x',D''_{(z,f'',y'')}},|(f'',y'') - ((f_{j,J})_J,y_j)|\}
  < \frac{\kappa}{2} T^{-\theta'\lambda'}.
\end{equation*}

Note that
 $\theta'\lambda' = \theta' \min \{1,\delta\} \lambda\ge \theta\lambda$.
We may also assume $\theta\le \delta$.
Thus by (\ref{eq:distxy2}) the distance of $x$ to $D''_{(z,f'',y'')}$ is strictly less than
$\kappa T^{-\theta\lambda}$. Similarly, $|y''-y| \le |y''-y_j|+|y_j-y|
< \kappa T^{-\theta\lambda}$. 
This yields (ii). 
The non-empty fibers $D''_{(z,f'',y'')}$ have dimension at most $r'$ and are
\qacs.
We are allowed to add the $D''$ to our collection in (i). 
Thus  \statement{$r$} is established  if $r'\le r-1$.

Second, say $r' = r$. 
In this case we verify that $D$ satisfies the properties from (i).
Recall that $D$ is member of a cell partition 
of $Z'$.
A   fiber of $D$ above $\IR^{m'}$ is either 
empty or a cell of dimension $r$. 

 We claim that $D$ is not among the cells
in the partition of $Z''$. 
 Indeed, otherwise we would have $D\subset Z''$.
By (\ref{eq:defZprimeprime}) the jacobian matrix attached to the $f_{j,J}$ has rank strictly
less than $n-r$ on the fibers of $D$. 
By construction 
each $f_{j,J}$ is irreducible as $J$ runs through $\cJ_{n,r}$
Thus Lemma 
\ref{lem:singulardim} contradicts
the fact that the fiber $D_{((f_{j,J})_J,y_j)}$ has dimension $r$.

For any $((f_J)_{J\in \cJ_{n,r}},y') \in \IR^{m'}$ we have $D_{((f_J)_J,y')}\subset
A=\zeroset{(f_J)_{J}}$. As the algebraic set on the right has
dimension at most $r$  
we have
 $\dim A=r$ if
$D_{((f_J)_J,y')}\not=\emptyset$. 
In this case and since $D\subset Z'\ssm Z''$, 
the jacobian matrix attached to $(f_J)_J$ has rank at least $n-r$ at
all points of 
$D_{((f_J)_J,y')}$.
So $D_{((f_J)_J,y')}$ is a \qac{} by
Lemma \ref{lem:cellinregularset}.

Now  
\begin{equation*}
  D_{((f_{J})_J,y')} \subset Z'_{((f_J)_J,y')} \subset \overline{C}_{y'} \subset Z_{y'}
\end{equation*}
by (\ref{eq:defineZprime}) and
as the compact set $Z$  contains $C$ and hence its closure $\overline
C$ in $\IR^m\times\IR^n$. 
Thus we can add $D$ to the cells mentioned in (i). 

As
only many finitely cells appear in the partition
of $Z'$, we get at
most finitely many cells by this process. We  already assumed
$\theta \le \delta$. So  
$x$ has distance strictly less than $\frac{\kappa}{2}T^{-\theta\lambda} < \kappa
T^{-\theta\lambda}$ to $D_{(f_j,y_j)}$  by
(\ref{eq:distxy2}). Moreover, $|y-y_j|<\kappa T^{-\theta\lambda}$ by the
same inequality. This
completes the proof that \statement{$r$} holds true.
\end{proof}

\begin{theorem}
\label{thm:approxfamily}
 Let $m\in\IN_0,n,\degvar\in\IN,\epsilon > 0, \kappa\in (0,1]$ and suppose
  $Z\subset\IR^m\times\IR^n$ is  compact and definable. 
  There exist $c=c(Z,\degvar,\epsilon,\kappa)\ge 1$, 
$\theta = \theta(Z,\degvar,\epsilon) \in (0,1]$,
integers
$l_1,\ldots,l_t\in\IN_0$, and bounded cells $D_j \subset
  \IR^{l_j}\times\IR^m\times\IR^n$ for all $j\in \{1,\ldots,t\}$ with the
  following properties: 
  \begin{enumerate}
  \item [(i)] Say $D = D_j$ for some $j\in \{1,\ldots,t\},
    z\in\IR^{l_j},$ and $y\in\IR^m$. Then $D_{(z,y)}\subset Z_y$ and
    if $D_{(z,y)}\not=\emptyset$, then 
$D_{(z,y)}$ is a \qac.
\item[(ii)] If $\lambda \ge \theta^{-1}, y\in\IR^m,$ and $T\ge 1$
  there exist an integer $N\ge 1$ with 
$N\le c T^\epsilon$ and 
$(j_p,z_p,y_p) \in \{1,\ldots,j\}\times\IR^{l_{j_p}}\times\IR^m$ 
for $p\in \{1,\ldots,N\}$ such that if 
\begin{equation*}
x\in Z_y \text{ and }q\in\IQ^n(T,\degvar)
 \text{ with }|x-q|<c^{-1}T^{-\lambda}
\end{equation*}
then $\dists{x,(D_{j_p})_{(z_p,y_p)}} < \kappa T^{-\theta\lambda}$ for some $p\in
\{1,\ldots, N\}$ with $|y-y_p|<\kappa T^{-\theta\lambda}$. 
  \end{enumerate}
\end{theorem}
\begin{proof}
We may assume $\epsilon\le 1$. 
The theorem then follows from \statement{$r$} ($0\le r\le n$) and
since $Z$ admits a partition 
into a finite number of cells. 
\end{proof}

We now extend this theorem to more general families of definable sets.
 To do this we introduce  the semi-algebraic homeomorphism $\varphi
:(-1,+\infty)^n\rightarrow (-\infty,1)^n$ given by
\begin{equation*}
  \varphi(x_1,\ldots,x_n) =
  \left(\frac{x_1}{1+x_1},\ldots,\frac{x_n}{1+x_n}\right)
\end{equation*}
with inverse
\begin{equation*}
  \varphi^{-1}(x_1,\ldots,x_n) = 
  \left(\frac{x_1}{1-x_1},\ldots,\frac{x_n}{1-x_n}\right).
\end{equation*}
If $x,x'\in[-1/2,+\infty)^n$, then 
$|\varphi(x)-\varphi(x')|\le 4 |x-x'|$ and if
$x,x'=(x'_1,\ldots,x'_n)\in(-\infty,1)$,
  then
  \begin{equation}
\label{eq:invarphibound}
    |\varphi^{-1}(x)-\varphi^{-1}(x')| \le
    \frac{|x-x'|}{\min_{1\le i\le n}\{1-x_i\}\min_{1\le i\le n} \{1-x'_i\}}.
  \end{equation}
 The map is not height-invariant but still satifies
\begin{equation*}
  H(\varphi(x))  \le 2 H(x)^2
\end{equation*}
for all algebraic $x\in \IR^n$ by basic height properties, cf. (\ref{eq:heightprops}). 
So $\varphi$ maps $\IQ^n(T,\degvar)$ to $\IQ^n(2T^2,\degvar)$.

\begin{lemma}
\label{lem:preimagevarphi}
  Suppose $D\subset\IR^n$ is a \qac{} of dimension $r$ with 
$D\subset (-\infty,1)^n$. Then
  $\dim\varphi^{-1}(D)=r$ and  $\varphi^{-1}(D)$ is an open
  subset of the non-singular locus of an $r$-dimensional  real algebraic set.
\end{lemma}
\begin{proof}
We have $\dim\varphi^{-1}(D)=r$ sind $\varphi$ is a homeomorphism.

  For a non-zero  $f\in \IR[X_1,\ldots,X_n]$ we set
  \begin{equation*}
 f^* = f\left( \frac{X_1}{1+X_1},\ldots,\frac{X_n}{1+X_n}\right) \prod_{j=1}^n
  (1+X_j)^{\deg_{X_j}(f)}   
  \end{equation*}
which is again a polynomial in $\IR[X_1,\ldots,X_n]$, we also set
$f^*=0$ if $f=0$.
 If $f$ vanishes on $A$, then $f^*$
vanishes on $B$, the Zariski closure of $\varphi^{-1}(A)=\varphi^{-1}(A\cap (-\infty,1)^n)$.

By hypothesis, there exists a  real algebraic set $A\subset\IR^n$ of
dimension $r$ such that $D$ is an open subset of $A\ssm
\sing{A}$. 
So $\dim A\ge \dim\varphi^{-1}(A) \ge \dim\varphi^{-1}(D) = r =\dim
A$.
Proposition 2.8.2 \cite{BCR:RAG} implies $\dim B = \dim
\varphi^{-1}(A)$ and so
 $\dim B = r$. 

We want to apply Lemma \ref{lem:cellinregularset}. First, we observe
 that $\varphi^{-1}(D)$, being homeomorphic to the cell $D$, is
 homeomorphic to an open subset of $\IR^{r}$. Say $x\in
 \varphi^{-1}(D)$, then $\varphi(x)\in D\subset A\ssm \sing{A}$. 
There are $f_1,\ldots,f_{n-r}\in
\IR[X_1,\ldots,X_n]$ that vanish on $A$ such 
that $(\frac{\partial f_i}{\partial X_j})_{1\le i\le n-r,1\le j\le n}$
has rank $n-r$ when evaluated at $\varphi(x)$.  
By the chain rule
$(\frac{\partial f^*_i}{\partial X_j})_{1\le i\le n-r,1\le j\le n}$
also has rank $n-r$ at $x$. 
We apply Lemma
\ref{lem:cellinregularset}  to $\varphi^{-1}(D)$, $B,$ and
$f^*_1,\ldots,f^*_{n-r}$ to find that $\varphi^{-1}(D)$ lies open in
$B\ssm \sing{B}$, as desired. 
\end{proof}

\begin{proof}[Proof of Theorem \ref{thm:approxfamilyclosed}]
After splitting up into the $2^n$ orthants of $\IR^n$ and switching signs we
 may assume   $Z\subset\IR^m\times[0,+\infty)^n$.

We consider the closure $Z'$ of the image of $Z$
 under $\textrm{id}_{\IR^m} \times\varphi$. This is a compact subset
 of $\IR^m\times [0,1]^n$. 
Since $Z$ is closed we mention
\begin{equation}
\label{eq:imageZ}
\bigl( \IR^m \times [0,1)^n\bigr) \cap Z' \subset (\textrm{id}_{\IR^m}\times\varphi)(Z)
\end{equation}
for later reference.

Say $ y\in\IR^m$ and  $T\ge 1$ such that there are $q\in \IQ^n(T,\degvar)$ 
and $x\in Z_y$ with $|x-q|<c^{-1}T^{-\lambda}$. Here and below
$c\ge 2$ is sufficiently large and $\theta$ is sufficiently small in terms of
the given data.
Moreover, we set $\kappa = 2^{-2\degvar-2}$.

We write $q=(q_1,\ldots,q_n)$ and $x=(x_1,\ldots,x_n)$. As $x_i\ge 0$
for all $i$, we find $q_i \ge x_i - c^{-1}T^{-\lambda}\ge -1/2$. So
$|\varphi(x)-\varphi(q)| \le 4|x-q|\le 4c^{-1}T^{-\lambda}$. 

For large $c$ and small $\theta$, by  Theorem \ref{thm:approxfamily}
applied to $Z',\degvar,\epsilon,$ and $\kappa$
we get $c'\ge 1, \theta' \in (0,1]$ and  $z'\in\IR^{l},y'\in\IR^m$ with
$\dists{\varphi(x),D_{(z',y')}}<\kappa T^{-\theta' \lambda}$ with at most
 $c'T^\epsilon$ possiblities for $(z',y')$,
$|y-y'|<\kappa T^{-\theta'\lambda}$ and  where
$D\subset \IR^l\times\IR^m \times \IR^n$ is a bounded cell. 
Moreover, only a finite number of $D$ appear and 
the fiber $D_{(z',y')}$ is a \qac.

There is an $x'' = (x''_1,\ldots,x''_n)\in D_{(z',y')}$ with 
$|\varphi(x)-x''|<\kappa T^{-\theta'\lambda}$. 
We want to show that $x'' \in [0,1)^n$.   
Observe that the entries of $x''$ are non-negative. 
We use  Liouville's Inequality to show that for any $i$ we have
  \begin{alignat*}1
1-\frac{q_i}{1+q_i} &= \frac{1}{1+q_i} \ge \frac{1}{H(1+q_i)^\degvar}
\ge \frac{1}{2^{\degvar} T^{\degvar}}
  \end{alignat*}
using again (\ref{eq:heightprops}) and $q\in\IQ^n(T,\degvar)$. 
Note 
that $|\varphi(q)-x''| <
4c^{-1}T^{-\lambda} + \kappa T^{-\theta'\lambda}$ and hence
\begin{alignat}1
\label{eq:xprimeprimebound}
  x''_i < \frac{q_i}{1+q_i} + \frac{4}{cT^{\lambda}}+\frac{\kappa}{T^{\theta'\lambda}}
\le 1-\frac{1}{2^{\degvar} T^{\degvar}} + \frac{4}{cT^{1/\theta}}+\frac{1}{2^{2\degvar+2} T^{\theta'/\theta}}
\le 1-\frac{1}{2^{\degvar+1}T^{\degvar}}<1
\end{alignat}
by our choice of $\kappa$, for large $c$ and small $\theta$. So $x''\in [0,1)^n$ as
  desired. Using a similar argument we find
  \begin{alignat}1
\label{eq:varphixbound}
    \frac{x_i}{1+x_i} \le \frac{q_i}{1+q_i} + \frac{4}{cT^{\lambda}}
\le 1 - \frac{1}{2^{\degvar}T^{\degvar}}  + \frac{4}{cT^{1/\theta}} \le
1-\frac{1}{2^{\degvar+1}T^{\degvar}}
  \end{alignat}
for large $c$ and small  $\theta$. 

Observe that $D\subset\IR^l\times\IR^m\times [0,1]$. So the
intersection $D\cap \IR^l\times\IR^m\times [0,1)^n$ is open in $D$. 
It is a finite union of cells
$D_1\cup\cdots\cup D_s$ with $\dim D_i = \dim D$ for all $1\le i\le s$
by Lemma \ref{lem:coveropenincell}.

Say $1\le i\le s$. The dimension of the cell $D_i$ equals the sum of the dimension of its
projection to $\IR^l\times\IR^m$ and the fiber dimension over $\IR^{l+m}$. 
The same holds for the cell $D\supset D_i$. Thus each $D_i$ has the same fiber
dimension as $D$ over $\IR^{l+m}$. We find that any
fiber of $D_i$ above a point in $\IR^l\times\IR^m$ lies open in the
respective fiber of $D$ by Lemma 1.14 in Chapter 4 \cite{D:oMin}. 
Therefore, all non-empty fibers of $D_i$ are \qacs{}. 

The first inequality in 
\begin{equation*}
  |x-\varphi^{-1}(x'')| = 
|\varphi^{-1}(\varphi(x))-\varphi^{-1}(x'')|\le 
2^{2\degvar + 2}T^{2\degvar} {|\varphi(x)-x''|}
<  2^{2\degvar +2}\kappa T^{2\degvar-\theta'\lambda}=T^{2\degvar-\theta'\lambda}
\end{equation*}
follows from (\ref{eq:xprimeprimebound}) and (\ref{eq:varphixbound})
applied to (\ref{eq:invarphibound}). 

Now $\varphi^{-1}(x'')$ lies in the preimage
$\varphi^{-1}((D_i)_{(z',y')})$. 
The distance of $x$ to this preimage is strictly less than
$T^{2\degvar-\theta'\lambda}$. 
Observe that $2\degvar-\theta' \lambda \le -\theta\lambda$
for small $\theta$ as $\lambda\ge \theta^{-1}$. 
So the distance is strictly less than 
  $T^{-\theta \lambda}$.

Observe that $(D_i)_{(z',y')}\subset (-\infty,1)^n$ and 
 $\varphi^{-1}((D_i)_{(z',y')})$ is connected as $\varphi$
is a homeomorphism. 
By Lemma \ref{lem:preimagevarphi}
this preimage
 satisfies the conditions in (i) of the
assertion. By (\ref{eq:imageZ}) the preimage lies in the respective
fiber
$Z_{y'}$ of $Z$. This completes the proof. 
\end{proof}

\section{Proof of Theorems  \ref{thm:approx2},
  \ref{thm:familybasicversion}, \ref{thm:approx}, and
  \ref{thm:approxqac}}
\label{sec:proofs}

\begin{proof}[Proof of Theorem \ref{thm:approxqac}]
The theorem follows from Theorem \ref{thm:approxfamilyclosed} applied to
the trivial family $Z=X\subset\IR^n$ where $m=0$.
\end{proof}

\begin{proof}[Proof of Theorem \ref{thm:approx2}]
By Northcott's Theorem $\IQ^n(2^{2+\degvar^2},\degvar)$ is
finite. So we may assume $T\ge
2^{2+\degvar^2}$ without loss of generality.

Let $c>0$ and $3\degvar^2 \theta$ be as in  Theorem \ref{thm:approxqac} applied
to $X,\degvar,$ and $\epsilon$.
Let $\lambda \ge \theta^{-1}$. 

Suppose $q\in\IQ^n(T,e)$ does not lie in $\nbhd{\alg{X},T^{-\theta\lambda}}$ and that there is 
$x\in X$ with $|x-q|<T^{-\lambda}$. Then $x'\in D$ and
$|x'-x|<T^{-3\degvar^2 \theta\lambda}$ for one among at most $cT^{\epsilon}$ sets
$D\subset X$ as in (i) of Theorem \ref{thm:approxqac}. As $3\degvar^2 \theta \le
1$ we have 
\begin{equation}
\label{eq:distqx}
 |x'-q|\le |x'-x|+|x-q|<T^{-3\degvar^2\theta\lambda}+T^{-\lambda} 
\le 2T^{-3\degvar^2 \theta\lambda}. 
\end{equation}

If $\dim D\ge 1$, then $\alg{D}=D$ by Lemma \ref{lem:Dalg} and so
$x'\in \alg{X}$. 
Recall $T\ge 2$ and $\theta\lambda\ge 1$.
Hence
 (\ref{eq:distqx}) imples $|x'-q|<2T^{-3\degvar^2 \theta\lambda}\le T^{-\theta\lambda}$ and this 
contradicts $q\not\in \nbhd{\alg{X},T^{-\theta\lambda}}$. 

Therefore, $\dim D=0$ and thus $D=\{x'\}$ as $D$ is connected.  Now
suppose a second $q'\in\IQ^n(T,\degvar)$ with $q'\not=q$  also satisfies
$|x'-q'|<2T^{-3\degvar^2 \theta\lambda}$. 
Then $|q'-q|<4T^{-3\degvar^2 \theta\lambda}$. As $q'\not=q$ Liouville's Inequality gives
$|q'-q|\ge (2H(q')H(q))^{-\degvar^2}\ge 2^{-\degvar^2}T^{-2\degvar^2}$. Therefore, 
$T^{3\degvar^2 \theta\lambda - 2\degvar^2}<2^{2+\degvar^2}$. But this contradicts 
$\theta\lambda\ge 1$ and $T\ge 2^{2+\degvar^2}$. 

We have shown that at most one  algebraic point of height at most $T$
and degree at most $\degvar$
approximates a singleton $D$. Thus the number of $q$ in question is at
most $cT^\epsilon$. 
\end{proof}

\begin{proof}[Proof of Theorem \ref{thm:familybasicversion}]
Instead of applying Theorem \ref{thm:approxqac}  as before we require
Theorem \ref{thm:approxfamilyclosed}, which holds for families, directly. The proof 
 is then very similar to the proof of Theorem \ref{thm:approx2}. 
\end{proof}

\begin{proof}[Proof of Theorem \ref{thm:approx}]
We use  Theorem \ref{thm:approxqac}. Indeed, any $x$ as in the set on the left of
(\ref{eq:counting}) lies in $\nbhd{(D_j)_z, T^{-\theta\lambda}}$ for one of at
most $cT^\epsilon$ sets $(D_j)_z$ as in (i) of Theorem
\ref{thm:approxqac}. 

If one particular $(D_j)_z$ has positive dimension, then it equals its
algebraic locus by Lemma \ref{lem:Dalg}. 
 In particular, $x\in
\nbhd{\alg{X},T^{-\theta\lambda}}$, which is impossible. 
So  $(D_j)_z$ has dimension
$0$ and, being connected, is  a singleton. This yields   (\ref{eq:counting})
when taking the $x_i$ to be the points appearing in the $(D_j)_z$. 
\end{proof}

\section{Application to Sums of Roots of Unity}
\label{sec:apps}

\begin{proof}[Proof of Theorem \ref{thm:pvaries}]
Our proof is by induction on $n$,   the statement being elementary if $n=1$. So say $n\ge 2$ and  let 
\begin{equation*}
  X = \left\{(x_1,\ldots,x_n) \in [0,1]^n : a_0 +  a_1 e^{2\pi \sqrt{-1}
    x_1} + \cdots + a_n e^{2\pi \sqrt{-1} x_n} = 0 \right\}
\end{equation*}
which is compact and definable in the polynomially bounded o-minimal structure $\IRan$.

We will choose $c$ and $\lambda$ in the argument below. 
Say $\zeta_j = e^{2\pi \sqrt{-1} q_j}$ with $q_j\in  \frac 1p \IZ\cap [0,1)$
such that
$0<|a_0+a_1\zeta_1+\cdots + a_n\zeta_n|< c^{-1}p^{-\lambda}$ and where
$p\le T$ is  a prime.
We may assume $p>T^\epsilon$ as there are at most $T^\epsilon$ primes
bounded by $T^\epsilon$.  
So
$|a_0+a_1\zeta_1+\cdots + a_n\zeta_n|< c^{-1}T^{-\epsilon\lambda}$.

If $c$ is large enough in terms of $(a_0,\ldots,a_n)$, then at least one among
$\zeta_1,\ldots,\zeta_n$ has order $p$.
For large $c$ the \L ojasiewicz Inequality  from Theorem
\ref{thm:lojasiewicz}  implies
\begin{equation*}
  \dists{q,X}<  T^{-\epsilon\lambda\delta}
\end{equation*}
where $q = (q_1,\ldots,q_n)$ and where  $\delta>0$  depends only on $X$.

We suppose $\epsilon\lambda\delta\ge \theta^{-1}$ 
 with $\theta$ from  
 Theorem \ref{thm:approx2} applied to $X,\degvar=1$, and $\epsilon$. There are two cases. 

In the first case $q$ is 
not in the $T^{-\theta\epsilon\lambda\delta}$-neighborhood around $\alg{X}$. 
As $p$ divides the denominator of $q$  it is among at most
$cT^\epsilon$ 
possibilities and we are done in this case. 

In the second case there is $x'=(x'_1,\ldots,x'_n)\in \alg{X}$ with
$|q-x'|<T^{-\theta\epsilon\lambda\delta}$. 

The locus $\alg{X}$ plays an important role in Zannier's proof
strategy
of the
Manin-Mumford Conjecture presented in his joint work with Pila \cite{PilaZannier}. Indeed, it is a well-known consequence of 
 Ax's Theorem, Corollary 2  \cite{AxSchanuel}, that
a non-trivial subsum 
\begin{equation*}
 a_0 + \sum_{j\in J} a_j e^{2\pi \sqrt{-1} x'_j}=0
\end{equation*}
vanishes for some
non-empty set $J \subsetneq \{1,\ldots,n\}$.
 The corresponding sum over
coordinates of $q$ must be small, i.e.
\begin{equation*}
  \left| a_0 +\sum_{j\in J} a_j \zeta_j\right| 
\le 2\pi n \max_{1\le j\le n} \{|a_j|\}|q-x'|
< c' T^{-\theta\epsilon\lambda\delta} 
\end{equation*}
where $c'>0$ depends only on $(a_1,\ldots,a_n)$. 

Let $\lambda'$ be the maximal value of $\lambda$ for this theorem applied by
induction to a sum involving at most $n-1$ roots of unity
and a subset of the $a_0,\ldots,a_n$ as coefficients. 
We may assume
$\theta\epsilon\lambda\delta \ge 1 + \lambda'$ and if $c''$ comes from
this theorem applied by induction we may also assume
that $T\ge c' c''$. Hence
\begin{equation*}
    \left| a_0 +\sum_{j\in J} a_j \zeta_j\right| < {c''}^{-1}
    T^{-\lambda'} \le {c''}^{-1} p^{-\lambda'}.
\end{equation*}

Say $a_0+\sum_{j\in J} a_j\zeta_j\not=0$. Then
by induction   there are at most $cT^\epsilon$ possibilites 
for $p$, if $c$ is sufficiently large. 

Finally, if $a_0+\sum_{j\in J} a_j\zeta_j=0$, then 
$\sum_{j\in I} a_j\zeta_j\not=0$ where
$I = \{1,\ldots,n\}\ssm J$. Say $j_0\in I$, then
\begin{equation*}
0<  \left|a_{j_0}+\sum_{j\in I\ssm\{j_0\}} a_j \zeta_j
  \zeta_{j_0}^{-1}\right|
= \left|a_0 +\sum_{j=1}^n a_j\zeta_j\right|< c^{-1}p^{-\lambda}. 
\end{equation*}
then, again by
induction on $n$, we conclude the claim if $\lambda\ge \lambda'$
and if $c$ is large enough. 
\end{proof}
\bibliographystyle{amsplain}  
\bibliography{literature}

\vfill
\medskip

\end{document}